\documentclass{siamart171218}
\usepackage[utf8]{inputenc}
\usepackage{amsmath,amssymb}
\usepackage{bm}
\usepackage{graphicx}
\usepackage[normalem]{ulem}

\usepackage{enumitem}
\usepackage{multirow}

\newcommand{\curl}{\nabla \times}
\renewcommand{\div}{\nabla \cdot}
\newcommand{\grad}{\nabla}

\newcommand{\cA}{\mathcal{A}}

\DeclareMathOperator{\Ran}{Ran}
\DeclareMathOperator{\Ker}{Ker}
\DeclareMathOperator{\Tr}{Tr}

\newcommand{\enorm}[1]{{\left\vert\kern-0.1ex\left\vert\kern-0.1ex\left\vert #1 \right\vert\kern-0.1ex\right\vert\kern-0.1ex\right\vert}}

\definecolor{mycol}{rgb}{1, 0.0, 0.5}

\newtheorem{remark}{Remark}[section]

\numberwithin{equation}{section}
\numberwithin{theorem}{section}

\title{Mixed and multipoint finite element methods for rotation-based poroelasticity\thanks{Submitted to the editors DATE.\funding{This project has received funding from the European Union's Horizon 2020 research and innovation programme under the Marie Skłodowska-Curie grant agreement No. 101031434 -- MiDiROM.}}}
\author{Wietse M. Boon\thanks{MOX Modeling and Scientific Computing, Politecnico di Milano, Italy ({\email{wietsemarijn.boon@polimi.it}}).}
\and Alessio Fumagalli\footnotemark[2] 
\and Anna Scotti\footnotemark[2]}

\begin{document}

\maketitle

\begin{abstract}
    This work proposes a mixed finite element method for the Biot poroelasticity equations that employs the lowest-order Raviart-Thomas finite element space for the solid displacement and piecewise constants for the fluid pressure. The method is based on the formulation of linearized elasticity as a weighted vector Laplace problem. By introducing the solid rotation and fluid flux as auxiliary variables, we form a four-field formulation of the Biot system, which is discretized using conforming mixed finite element spaces. The auxiliary variables are subsequently removed from the system in a local hybridization technique to obtain a multipoint rotation-flux mixed finite element method. Stability and convergence of the four-field and multipoint mixed finite element methods are shown in terms of weighted norms, which additionally leads to parameter-robust preconditioners. Numerical experiments confirm the theoretical results.
\end{abstract}



\section{Introduction}

The cornerstone of this work is a reformulation of the linearized elasticity equations as a weighted vector-Laplacian on the displacement. The key is then to consider the mixed formulation of this {problem} by introducing the solid rotation as an auxiliary variable. This allows us to relieve the $H^1$-regularity requirement that is typically placed on the displacement variable.

Rotation-based formulations of elasticity and poroelasticity were recently investigated in \cite{anaya2019mixed} and \cite{anaya2020rotation}, {where} the displacement is sought in $H^1$ and the rotation in $L^2$. In contrast, we seek the displacement in the larger space $H(\div, \Omega)$ and the rotation in the smaller space $H(\curl, \Omega)$. In turn, our variational problem requires its own, distinct \emph{a priori} analysis and leads to {a} different {choice of the} finite element spaces.

An important advantage of the four-field formulation we consider is that it allows for mass-lumping techniques that are common to multipoint mixed finite element methods. The primary example of these methods is the \emph{multipoint flux mixed finite element method (MF-MFEM)} \cite{wheeler2006multipoint}. The method, {traditionally employed for elliptic problems such as Darcy's flow,} employs the Brezzi-Douglas-Marini \cite{brezzi1985two} space to model the Darcy flux and introduces a low-order quadrature rule on its inner product to obtain an approximation of the mass matrix that is easily invertible. In turn, the flux variable can be eliminated locally, leading to a scheme with cell-centered pressures that is closely related to the multipoint flux approximation (MPFA) finite volume method \cite{aavatsmark2002introduction}.

The MF-MFEM was extended to the case of linearized elasticity as the \emph{multipoint stress MFEM (MS-MFEM)} \cite{ambartsumyan2020multipoint}, which has in turn been applied to the Biot equations \cite{ambartsumyan2020coupled} and Stokes flow \cite{caucao2022multipoint}.
From the perspective of exterior calculus, these hybridization techniques were recognized as a way to compute local coderivatives and subsequently generalized to a larger class of mixed finite element spaces \cite{lee2018local}. Based on these results, the \emph{multipoint vorticity MFEM (MV-MFEM}) was recently developed for a vorticity-velocity-pressure formulation of Stokes \cite{boon2022multipoint}.

The similarities between the Stokes and the Biot equations were used in \cite{rodrigo2018new} to construct stabilized mixed finite element discretizations.
Through a similar observation, we herein extend MV-MFEM for Stokes \cite{boon2022multipoint} to a \emph{multipoint rotation mixed finite element method (MR-MFEM}) for linearized elasticity and poroelasticity. As mentioned before, we obtain the MR-MFEM from a rotation-based formulation of elasticity. It therefore employs the Raviart-Thomas ($\mathbb{RT}_0$) space for the solid displacement, as opposed to piecewise constants ($\mathbb{P}_0^n$) in MS-MFEM. The coupling with fluid flow in a poroelastic setting is then naturally incorporated because the divergence of the displacement is well-defined on $\mathbb{RT}_0$.


The paper is organized as follows. First, the conventions concerning notation are introduced in \Cref{sub: Preliminaries and notation} and the model problem is presented in \Cref{sec:model_problem}.
The ensuing sections contain the main contributions of this work, which are as follows:
\begin{itemize}
    \item Well-posedness analysis of a rotation-based, four-field formulation of poroelasticity with the solid rotation in $H(\curl, \Omega)$ and displacement in $H(\div, \Omega)$, using parameter-weighted norms (\Cref{sec:analysis}).
    \item Stability and convergence analysis of two families of mixed finite elements that conform to the four-field formulation (\Cref{sec:discretization}).
    \item A stable and convergent multipoint rotation-flux mixed finite element method for the Biot equations that employs the $\mathbb{RT}_0 \times \mathbb{P}_0$ pair for the solid displacement and fluid pressure (\Cref{sec:a_multipoint_rotation_flux_mfe_method}).
    \item Parameter-robust preconditioners for the mixed finite element methods based on the analysis in weighted norms (\Cref{sec: Parameter-robust preconditioning}).
    \item Numerical experiments that confirm the convergence of the methods and robustness of the preconditioner (\Cref{sec: Numerical results}).
\end{itemize}
Concluding remarks are given in \Cref{sec:conclusions}.

\subsection{Preliminaries and notation}
\label{sub: Preliminaries and notation}

Let $\Omega \subset \mathbb{R}^n$ be a contractible, Lipschitz domain with $n \in \{2, 3\}$. Let $L^2(\Omega)$ be the space of square integrable functions on $\Omega$ and let its inner product be denoted by $\langle \cdot, \cdot \rangle_{\Omega}$. The space $L^2(\Omega)$ is endowed with the norm $\| \cdot \| := \sqrt{\langle \cdot, \cdot \rangle_{\Omega}}$. We {apply the same} notation for the inner product and norm of square-integrable vector functions in $(L^2(\Omega))^n$.

Let $H(\div, \Omega)$ be the subspace of $(L^2(\Omega))^n$ that contains functions with square integrable divergence. Let $\curl$ denote the conventional curl operator in 3D. {Note that} for $n = 2$, {we simply have that}  $\curl r := (- \partial_2 r, \partial_1 r)$ and $\curl u := \partial_2 u_1 - \partial_1 u_2$ for sufficiently regular scalar fields $r$ and vector fields $u$. In turn, let $H(\curl, \Omega)$ be the subspace of $(L^2(\Omega))^{k_n}$, with $k_n := \left(\begin{smallmatrix} n \\ 2 \end{smallmatrix}\right)$, consisting of functions with square integrable curl.

For a Hilbert space $X$, let $X'$ denote its dual, and the corresponding duality pairing is given by $\langle \cdot, \cdot \rangle_{X' \times X}$. For notational brevity, we omit the subscript on duality pairings since it can be deduced from context. Given an operator $A: Z \to Y$ and a subspace $X \subseteq Z$, let $A|_X$ be the restriction of $A$ on $X$. We denote the kernel and range of the restriction by $\Ker(A, X) := \Ker(A|_X)$ and $\Ran(A, X) := \Ran(A|_X)$, respectively.

For $a, b \in \mathbb{R}$, the notation $a \lesssim b$ implies that a constant $C > 0$ exists, independent of material or discretization parameters, such that $C a \le b$.
However, the constant $C$ may depend on the domain $\Omega$ and on the shape-regularity of the mesh.
We use $\gtrsim$ analogously and $a \eqsim b$ if and only if $a \lesssim b \lesssim a$.

\section{Model Problem}
\label{sec:model_problem}

We start in \Cref{sub: elasticity} with linear elasticity to highlight the manipulation of the system that lies at the heart of the proposed numerical methods. The coupling to flow is introduced afterward in \Cref{sub: flow coupling}.

\subsection{Linearized elasticity as a weighted vector-Laplacian}
\label{sub: elasticity}

Let us consider the governing equations of linearized elasticity in terms of the Cauchy stress $\sigma$ and displacement $u$:
\begin{align}
    \sigma &= 2\mu \varepsilon (u) + (\lambda \div u) I, &
    - \div \sigma &= f_u.
\end{align}
Here, $\mu$ and $\lambda$ are the Lam\'e parameters, $f_u$ is a body force, $\varepsilon$ is the symmetric gradient and $I \in \mathbb{R}^{n \times n}$ the identity tensor.
We reformulate these equations by recalling the following calculus identity:
\begin{align}
    - \div \varepsilon ( u )
    = \frac12 \curl (\curl u)
    - \grad (\div u).
\end{align}
We substitute this identity and the definition of $\sigma$ in the momentum balance equation and assume spatially constant $\mu$ to obtain
\begin{align} \label{eq: momentum balance}
    \curl \left(\mu \curl u \right)
    - \grad (2\mu + \lambda) \div u &= f_u
\end{align}

We now define the rotation variable $r := \mu \curl u$, which leads us to the strong form of our 2-field formulation for linear elasticity in terms of $(r, u)$:
\begin{subequations}\label{eq:mechanics}
\begin{align}
    \mu^{-1} r - \curl u &= 0, &
    \curl r - \grad(2\mu + \lambda) \div u &= f_u.
\end{align}
subject to the boundary conditions
\begin{align}
    \begin{aligned}
    \nu \cdot u &= 0, &
    \nu \times r &= 0 , &
    \text{on } \partial_r \Omega, \\
    \nu \times u &= \nu \times u_0, &
    (2\mu + \lambda) \Tr\varepsilon(u) &= \sigma_0 , &
    \text{on } \partial_u \Omega.
    \end{aligned}
\end{align}
\end{subequations}
Here, $\partial_r \Omega \cup \partial_u \Omega$ is a disjoint decomposition of the boundary $\partial \Omega$ and $\nu$ is its outward oriented, unit normal vector. To ensure uniqueness, we assume that the boundary decomposition is such that $\| \nu \cdot \phi \|_{\partial_r \Omega} + \| \nu \times \phi \|_{\partial_u \Omega} > 0$ for all non-zero rigid body motions $\phi$.


To derive the variational formulation of \eqref{eq:mechanics}, we introduce the following Hilbert spaces in which to seek the rotation and displacement variables:
\begin{align}\label{eqs: def R U}
    \begin{aligned}
    R &:= \left\{ r \in H(\curl, \Omega) \mid  \nu \times r = 0 \text{ on } \partial_r \Omega \right\}, \\
    U &:= \left\{ u \in H(\div, \Omega) \mid  \nu \cdot u = 0 \text{ on } \partial_r \Omega \right\}.
    \end{aligned}
\end{align}

By introducing test functions $(\tilde r, \tilde u) \in R \times U$ and using integration by parts, we obtain the variational formulation:
find $(\tilde r, \tilde u) \in R \times U$ such that
\begin{align} \label{eqs: weak elasticity cont}
    \begin{aligned}
    \langle \mu^{-1} r, \tilde r \rangle_\Omega
    - \langle u, \curl \tilde r \rangle_\Omega
    &= \langle u_0, \nu \times \tilde r \rangle_{\partial_u \Omega},
    & \forall \tilde r \in R, \\
    \langle \curl r, \tilde u \rangle_\Omega
    + \langle (2\mu + \lambda) \div u, \div \tilde u \rangle_\Omega &= \langle f_u, \tilde u \rangle_\Omega
    - \langle \sigma_0, \nu \cdot \tilde u \rangle_{\partial_u \Omega},
    & \forall \tilde u \in U.
    \end{aligned}
\end{align}

\begin{remark}
    The boundary conditions in \eqref{eq:mechanics} do not immediately {translate into classical boundary conditions such as} clamped boundaries, i.e. $u = 0$ on $\partial \Omega$. We refer to \cite{arnold2012mixed} for techniques that handle this case.
\end{remark}

\subsection{Coupling to porous medium flow}
\label{sub: flow coupling}

Next, we consider the setting of a poroelastic medium in which fluid flow and solid mechanics form a fully coupled system known as the quasi-steady Biot equations:
\begin{align}\label{eqs: strong form original}
    \begin{aligned}
    \sigma &= 2\mu \varepsilon (u) + (\lambda \div u - \alpha p) I, &
    - \div \sigma &= f_u, \\
    \check q &= - K \grad p , &
    \partial_t (c_0 p + \alpha \div u) + \div \check q &= \check f_p,
    \end{aligned}
\end{align}
with $\check q$ the Darcy flux and $p$ the fluid pressure. Moreover, $K$ is the hydraulic conductivity, $c_0$ the specific storativity, and $\alpha$ is the Biot-Willis constant.

Using the same steps as in \Cref{sub: elasticity}, we rewrite the elasticity equations in terms of rotation and displacement. As time discretization, we choose an implicit method such as the backward Euler or Crank-Nicolson scheme with time step $\Delta t$.

In order to obtain an advantageous scaling with the time step, we introduce the scaled Darcy flux $q := \delta \check q$, with $\delta := \sqrt{\Delta t}$. Different scalings of $\check q$ are possible, but may result in systems that either do not have favorable symmetries, which complicates the analysis, or contain negative powers of $\Delta t$, which can be undesirable in the case of small time steps.

By including {quantities relative to the}  previous time step in the right-hand side $f_p$, we obtain the semi-discrete, four-field formulation of the Biot equations:
\begin{subequations} \label{eqs: Biot strong cont}
\begin{align} \label{eq: Biot system}
    \begin{bmatrix}
        \mu^{-1} & - \curl \\
        \curl & - \grad (2\mu + \lambda) \div & & \grad \alpha \\
        & & K^{-1} & \delta \grad \\
        & \alpha \div & \div \delta & c_0
    \end{bmatrix}
    \begin{bmatrix}
        r \\ u \\ q \\ p
    \end{bmatrix}
    &=
    \begin{bmatrix}
        0 \\ f_u \\ 0 \\ f_p
    \end{bmatrix},
\end{align}
subject to the boundary conditions
\begin{align}
    \begin{aligned}
    \nu \cdot u &= 0, &
    \nu \times r &= 0 , &
    \text{on } \partial_r \Omega, \\
    \nu \times u &= \nu \times u_0, &
    (2\mu + \lambda) \Tr\varepsilon(u) - \alpha p &= \sigma_0 , &
    \text{on } \partial_u \Omega, \\
    p &= p_0,
    \text{  on } \partial_p \Omega, &
    \nu \cdot q &= 0, &
    \text{on } \partial_q \Omega.
    \end{aligned}
\end{align}
\end{subequations}
We assume that $\partial_p \Omega \cup \partial_q \Omega$ forms a disjoint decomposition of the boundary with $|\partial_p \Omega| > 0$.
To facilitate the analysis in the next section, we define the following Hilbert spaces for the fluid flux and pressure:
\begin{align} \label{eqs: def Q P}
    Q &:= \left\{ q \in H(\div, \Omega) \mid \nu \cdot q = 0 \text{ on } \partial_q \Omega \right\}, &
    P &:= L^2(\Omega).
\end{align}


\section{Analysis of the semi-discrete problem}
\label{sec:analysis}

We continue by constructing and analyzing the variational formulation of system \eqref{eqs: Biot strong cont}, which is continuous in space and discrete in time. As short-hand notation, we will use $x := (r, u, q, p)$ and $\tilde x := (\tilde r, \tilde u, \tilde q, \tilde p)$ belonging to the Hilbert space:
\begin{align}
    X := R \times U \times Q \times P,
\end{align}
with the spaces $R, U$ defined in \eqref{eqs: def R U} and $Q, P$ in \eqref{eqs: def Q P}.

For simplicity, let all material parameters be homogenous in space and we moreover assume that the conductivity $K$ is isotropic and thus given by a positive scalar.
To highlight the structure of the system, we define the operators $A: X \to X'$ and $B: X \to X'$ and the functional $f \in X'$ as follows:
\begin{align*}
    \langle Ax, \tilde x \rangle &:=
    \mu^{-1} \langle r, \tilde r \rangle_\Omega
    + (2\mu + \lambda) \langle \div u, \div \tilde u \rangle_\Omega
    + K^{-1} \langle q, \tilde q \rangle_\Omega
    + c_0 \langle p, \tilde p \rangle_\Omega,
    \\
    \langle Bx, \tilde x \rangle &:=
    \langle \curl r, \tilde u \rangle_\Omega
    +
    \langle \div(\alpha u + \delta q), \tilde p \rangle_\Omega, \\
    \langle f, \tilde x \rangle &:=
    \langle u_0, \nu \times \tilde r \rangle_{\partial_u \Omega}
    + \langle f_u, \tilde u \rangle_\Omega
    - \langle \sigma_0, \nu \cdot \tilde u \rangle_{\partial_u \Omega}
    - \langle p_0, \nu \cdot \delta \tilde q \rangle_{\partial_p \Omega}
    + \langle f_p, \tilde p \rangle_\Omega.
\end{align*}

Let the compound operator $\cA: X \to X'$ be such that
\begin{align}
    \cA &:= A + B - B^*,
\end{align}
with $B^*: X \to X'$ the adjoint of $B$ defined by $\langle B^* x, \tilde x \rangle := \langle B \tilde x, x \rangle$.

The variational formulation of \eqref{eqs: Biot strong cont} is then given by: find $x \in X$ such that
\begin{align} \label{eq: Biot weak cont}
    \langle \cA x, \tilde x \rangle &= \langle f, \tilde x \rangle,
    & \forall \tilde x &\in X.
\end{align}

Based on the scaling in the operator $\cA$, we endow $X$ with the following parameter-dependent norm:
\begin{align} \label{eq: norm X}
    \| x \|_X^2 := &\
    \mu^{-1} (\| r \|^2 + \| \curl r \|^2)
    + \mu \| u \|^2 + (2\mu + \lambda)\| \div u \|^2 \nonumber \\
    &+ K^{-1} \| q \|^2
    + \frac{\delta^2}{\eta + c_0} \| \div q \|^2
    + (\eta + c_0) \| p \|^2.
\end{align}
3Here, $\eta := \frac{\alpha^2}{2\mu + \lambda} + \delta^2 K$ is a scaling parameter particularly chosen for the ensuing analysis. In order for these norms to be well-defined, we assume that $\mu$, $\lambda$, $K$, and $(\eta + c_0)$ are positive. However, we do not explicitly bound these parameters away from zero, so that we do not rely on small, lower bounds in our analysis.

It is convenient to analyze problem \eqref{eq: Biot weak cont} using an equivalent energy norm, which we introduce in the following lemma.

\begin{lemma}[Energy norm] \label{lem: norm equiv}
    Let $\Pi$ be the $L^2$-projection onto $\Ran(\curl, R)$ and let the energy norm be given by
    \begin{align} \label{eq: enorm}
        \enorm{x}^2 := &\
        \mu^{-1} (\| r \|^2 + \| \curl r \|^2)
        + \mu \| \Pi u \|^2 + (2\mu + \lambda)\| \div u \|^2 \nonumber \\
        &+ K^{-1} \| q \|^2
        + \frac1{\eta + c_0} \| \div (\alpha u + \delta q) \|^2
        + (\eta + c_0) \| p \|^2.
    \end{align}
    Then the following equivalence holds:
    \begin{align}
        \| x \|_X &\eqsim \enorm{x}, &
        \forall x &\in X.
    \end{align}
\end{lemma}
\begin{proof}
    Let us consider the lower bound ``$\gtrsim$''. We first show that $\| u \| \gtrsim \| \Pi u \|$ follows immediately from the $L^2$-orthogonality of $\Pi$:
    \begin{subequations} \label{eqs: lower bound}
    \begin{align}
        \| u \|^2 = \| \Pi u \|^2 + \| (I - \Pi) u \|^2 \ge \| \Pi u \|^2.
    \end{align}
    Secondly, we consider the norms on the divergence terms. The triangle-type inequality $\| a  + b\|^2 \le 2(\| a \|^2 + \| b \|^2)$ and the lower bound $\eta + c_0 \ge \frac{\alpha^2}{2\mu + \lambda}$ yield
    \begin{align} \label{eq: lower bound 2}
        \frac1{\eta + c_0} \| \div (\alpha u + \delta q) \|^2
        &\le 2 \left( \frac{\alpha^2}{\eta + c_0} \| \div u \|^2 + \frac{\delta^2}{\eta + c_0} \| \div q \|^2 \right) \nonumber \\
        &\le 2 \left( (2\mu + \lambda) \| \div u \|^2 + \frac{\delta^2}{\eta + c_0} \| \div q \|^2 \right).
    \end{align}
    \end{subequations}
    Collecting \eqref{eqs: lower bound}, we conclude $\| x \|_X \gtrsim \enorm{x}$.

    For the upper bound ``$\lesssim$'', we once again first bound the norm $\| u \|$. The Poincar\'e inequality implies that a $\alpha_u > 0$ exists such that
    \begin{align*}
        \alpha_u \| u \| &\le \| \div u \|,
        & \forall u \perp \Ker(\div, U).
    \end{align*}
    Since $\Omega$ is contractible, we have $\Ker(\div, U) = \Ran(\curl, R)$ and we use the Poincar\'e inequality to derive
    \begin{subequations} \label{eqs: upper bound}
    \begin{align}
        \mu \| u \|^2
        &\le \mu \| \Pi u \|^2 + \frac1{\alpha_u^2} \mu \| \div ((I - \Pi) u) \|^2 \nonumber\\
        &= \mu \| \Pi u \|^2 + \frac1{\alpha_u^2} \mu \| \div u \|^2 \nonumber\\
        &\le \max\left\{ \frac1{2\alpha_u^2} , 1 \right\} \left( \mu \| \Pi u \|^2 + (2\mu + \lambda) \| \div u \|^2 \right).
    \end{align}

    It remains to show a bound on the divergence terms. Using the same triangle-type inequality and lower bound on $\eta + c_0$ as in \eqref{eq: lower bound 2}, we derive
    \begin{align*}
        \frac1{\eta + c_0} \| \div (\alpha u + \delta q) \|^2
        &\ge \frac1{\eta + c_0} \left( \frac12 \| \div \delta q \|^2 - \| \div \alpha u \|^2 \right)\\
        &\ge \frac12 \frac{\delta^2}{\eta + c_0} \| \div q \|^2 - (2\mu + \lambda) \| \div u \|^2.
    \end{align*}
    From this result, we deduce
    \begin{align}
        (2\mu + \lambda) \| \div u \|^2
        &+ \frac1{\eta + c_0} \| \div (\alpha u + \delta q) \|^2 \nonumber\\
        &\ge (2\mu + \lambda) \| \div u \|^2 + \frac23 \frac1{\eta + c_0} \| \div (\alpha u + \delta q) \|^2 \nonumber\\
        &\ge \frac13 \left( (2\mu + \lambda) \| \div u \|^2 + \frac{\delta^2}{\eta + c_0} \| \div q \|^2 \right).
    \end{align}
    \end{subequations}
    Collecting \eqref{eqs: upper bound}, we have $\| x \|_X \lesssim \enorm{x}$ and the result follows.
\end{proof}

With the energy norm defined {by \eqref{eq: enorm}}, we are now ready to prove the key prerequisites for well-posedness of \eqref{eq: Biot weak cont}, namely the continuity and inf-sup conditions on $\cA$.

\begin{lemma}[Continuity] \label{lem: continuity cont}
    The operator $\cA: X \to X'$ is continuous:
    \begin{align}
        \| \cA x \|
        &:= \sup_{\tilde x \in X}
        \frac{\langle \cA x, \tilde x \rangle}{\| \tilde x \|_X} \lesssim \| x \|_X,
        & \forall x \in X.
    \end{align}
\end{lemma}
\begin{proof}
    First, we apply the Cauchy-Schwarz inequality to each of the terms, e.g.
    \begin{align*}
        \mu^{-1} \langle r, \tilde r \rangle_\Omega
        &\le
        \mu^{-1} \| r \| \| \tilde r \|, &
        \langle \curl r, \tilde u \rangle_\Omega
        = \langle \curl r, \Pi \tilde u \rangle_\Omega
        &\le
        \frac1{\sqrt{\mu}} \| r \| \sqrt{\mu} \| \Pi \tilde u \|.
    \end{align*}
    After summing all products, we use the Cauchy-Schwarz once more to conclude that
    $\langle \cA x, \tilde x \rangle \lesssim \enorm{x} \ \enorm{\tilde x}$. The result in $\| \cdot \|_X$ follows from the norm equivalence shown in \Cref{lem: norm equiv}.
\end{proof}

\begin{lemma}[Inf-sup] \label{lem: infsup cont}
    The operator $\cA$ is bounded from below as
    \begin{align}
        \| \cA x \| &\gtrsim \| x \|_X,
        & \forall x \in X.
    \end{align}
\end{lemma}
\begin{proof}
    We aim to show that, for given $x \in X$, a test function $\tilde x \in X$ exists with the properties
    \begin{align} \label{eq: to prove for inf-sup}
        \langle \cA x, \tilde x \rangle &\gtrsim \enorm{x}^2, &
        \enorm{\tilde x} &\lesssim \enorm{x}.
    \end{align}
    The equivalence from \Cref{lem: norm equiv} will then provide the result in the norm $\| \cdot \|_X$.

    We construct this test function explicitly in four parts. First, let $\tilde x_0 := x = (r, u, q, p)$ for which we derive
    \begin{align} \label{eq: test function 0}
        \langle \cA x, \tilde x_0 \rangle &=
        \langle A x, x \rangle =
        \mu^{-1} \| r \|^2
        + (2\mu + \lambda) \| \div u \|^2
        + K^{-1} \| q \|^2
        + c_0 \| p \|^2.
    \end{align}

    The second test function we consider is given by $\tilde x_1 := (0, \mu^{-1} \curl r, 0, \frac1{\eta_0 + c_0} \div(\alpha u + \delta q))$ with $\eta_0 \ge 0$ to be chosen later. Using the orthogonality $\div \curl r = 0$, we derive using the Cauchy-Schwarz and Young's inequality:
    \begin{align} \label{eq: test function 1}
        \langle \cA x, \tilde x_1 \rangle &=
        \mu^{-1} \| \curl r \|^2
        + \frac1{\eta_0 + c_0} \| \div(\alpha u + \delta q) \|^2
        + \langle \frac{c_0}{\eta_0 + c_0} p, \div(\alpha u + \delta q) \rangle_\Omega
        \nonumber \\
        &\ge \mu^{-1} \| \curl r \|^2
        + \frac1{\eta_0 + c_0} \| \div(\alpha u + \delta q) \|^2
        - \| p \| \| \div(\alpha u + \delta q) \|
        \nonumber \\
        &\ge \mu^{-1} \| \curl r \|^2
        + \frac1{2(\eta_0 + c_0)} \| \div(\alpha u + \delta q) \|^2
        - \frac{\eta_0 + c_0}2 \| p \|^2.
    \end{align}

    The last two components are constructed by exploiting the inf-sup conditions on the operators that compose $B$. The first of these implies that a $\beta_r > 0$ exists such that for each $u \in U$, there exists a $\tilde r_u \in R$ with
    \begin{align*}
        \curl \tilde r_u &= \Pi u, &
        \beta_r \left( \| \tilde r_u \|^2 + \| \curl \tilde r_u \|^2 \right) &\le \| \Pi u \|^2.
    \end{align*}
    We use $\tilde r_u$ to define $\tilde x_2 := (- \beta_r \mu \tilde r_u, 0,0, 0)$. The Cauchy-Schwarz and Young inequalities give us:
    \begin{align} \label{eq: test function 2}
        \langle \cA x, \tilde x_2 \rangle &=
        \beta_r \mu \| \Pi u \|^2
        - \beta_r \langle r, \tilde r_u \rangle_\Omega
        \nonumber \\
        &\ge
        \beta_r \mu \| \Pi u \|^2
        - \frac12 \left(\mu^{-1}\| r \|^2 + \beta_r^2 \mu\| \tilde r_u \|^2 \right) \nonumber\\
        &\ge
        \frac{\beta_r}{2} \mu \| \Pi u \|^2
        - \frac12 \mu^{-1}\| r \|^2
    \end{align}

    The final test function is constructed similarly. Two constants $\beta_u, \beta_q > 0$ exist such that for each $p \in P$, there exists a pair $(\tilde u_p, \tilde q_p) \in U \times Q$ with the properties
        \begin{align*}
            \div \tilde u_p &= p, &
            \Pi \tilde u_p &=0 , &
            \beta_u \left( \| \tilde u_p \|^2 + \| \div \tilde u_p \|^2 \right) &\le \| p \|^2, \\
            \div \tilde q_p &= p, &
            \Pi \tilde q_p &=0 , &
            \beta_q \left( \| \tilde q_p \|^2 + \| \div \tilde q_p \|^2 \right) &\le \| p \|^2
        \end{align*}
    These allow us to define $\tilde x_3 := (0, - \frac{\alpha}{2\mu + \lambda} \tilde u_p, - \beta_q \delta K \tilde q_p, 0)$. Let us use the short-hand notation $\eta_0 := \frac12 \left(\frac{\alpha^2}{2\mu + \lambda} + \beta_q \delta^2 K \right)$. We then proceed as in \eqref{eq: test function 2} to derive
    \begin{align} \label{eq: test function 3}
        \langle \cA x, \tilde x_3 \rangle &=
        \left( \frac{\alpha^2}{2\mu + \lambda} + \beta_q \delta^2 K \right)\| p \|^2
        - \alpha \langle \div u, p \rangle_\Omega
        - \beta_q \delta \langle q, \tilde q_p \rangle_\Omega \nonumber\\
        &\ge
        2 \eta_0 \| p \|^2 \nonumber\\
        &\quad
        - \frac12 \left(
        (2\mu + \lambda) \| \div u \|^2 + \frac{\alpha^2}{2\mu + \lambda} \| p \|^2
        +  K^{-1} \| q \|^2 +  \beta_q^2 \delta^2 K \|\tilde q_p \|^2 \right) \nonumber\\
        &\ge
        \eta_0 \| p \|^2
        - \frac12
        (2\mu + \lambda) \| \div u \|^2
        - \frac12 K^{-1} \| q \|^2.
    \end{align}

    Letting $\tilde x := \sum_i \tilde x_i$, we sum \eqref{eq: test function 0}, \eqref{eq: test function 1}, \eqref{eq: test function 2}, and \eqref{eq: test function 3} to obtain
    \begin{align} \label{eq: lower bound on A}
        \langle \cA x, \tilde x \rangle &\ge
        \frac12 \Bigg(
        \mu^{-1} \| r \|^2
        + \frac2{\mu} \| \curl r \|^2
        + \beta_r \mu \| \Pi u \|^2
        + (2\mu + \lambda) \| \div u \|^2 \nonumber \\
        &\quad
        + \frac1{\eta_0 + c_0} \| \div(\alpha u + \delta q) \|^2
        + K^{-1} \| q \|^2
        + (\eta_0 + c_0) \| p \|^2
        \Bigg) \nonumber\\
        &\gtrsim \enorm{x}^2,
    \end{align}
    where we used $\eta_0 \eqsim \eta$ in the final inequality since $\min\{1, \beta_q\} \eta \le 2\eta_0 \le \max\{1, \beta_q\} \eta$.

    It remains to show that $\tilde x$ is bounded. Clearly, we have $\enorm{\tilde x_0} = \enorm{x}$, so we continue with the bounds for the remaining $x_i$:
    \begin{align}\label{eqs: bounds on test functions}
        \begin{aligned}
        \enorm{\tilde x_1}^2 &=
        \mu \| \mu^{-1}\curl r \|^2
        + \frac{\eta + c_0}{(\eta_0 + c_0)^2} \| \div (\alpha u + \delta q) \|^2 \\
        &\lesssim
        \mu^{-1} \| \curl r \|^2
        + \frac1{\eta + c_0} \| \div (\alpha u + \delta q) \|^2
        \le \enorm{x}^2 \\
        \enorm{\tilde x_2}^2 &=
        \mu \beta_r^2 (\| \tilde r_u \|^2 + \| \curl \tilde r_u \|^2)
        \lesssim \mu \| \Pi u \|^2
        \le \enorm{x}^2 \\
        \enorm{\tilde x_3}^2 &=
        \frac{\alpha^2}{2\mu + \lambda} \| p \|^2
        + \beta_q^2 \delta^2 K \| \tilde q_p \|^2
        + \frac1{\eta + c_0} \| 2\eta_0 p \|^2  \\
        &\lesssim
        \eta \| p \|^2
        \le \enorm{x}^2.
        \end{aligned}
    \end{align}

    Collecting \eqref{eqs: bounds on test functions}, we have $\enorm{\tilde x} \lesssim \enorm{x}$. Together with \eqref{eq: lower bound on A}, we have shown \eqref{eq: to prove for inf-sup} and the result follows by \Cref{lem: norm equiv}.
\end{proof}

\begin{theorem}[Well-posedness] \label{thm: well-posed cont}
    Problem \eqref{eq: Biot weak cont} admits a unique solution $x \in X$ that satisfies
    \begin{align} \label{eq: bound on solution}
        \| x \|_X \lesssim \| f \|_{X'} := \sup_{\tilde x \in X} \frac{\langle f, \tilde x \rangle}{\| \tilde x \|_X}.
    \end{align}
\end{theorem}
\begin{proof}
    We aim to utilize the Babu\v{s}ka-Lax-Milgram theorem. For this, we need to show that for each $x \in X$, a $\tilde x \in X$ exists such that
    $\langle \cA \tilde x, x \rangle > 0$.
    The symmetries of $\cA$ allow us to write:
    \begin{align} \label{eq: lower bound A*}
        \langle \cA \tilde x, x \rangle
        = \langle \cA (\tilde r, \tilde u, \tilde q, \tilde p), (r, u, q, p) \rangle
        = \langle \cA (r, -u, -q, p), (\tilde r, -\tilde u, -\tilde q, \tilde p) \rangle
    \end{align}
    In turn, we use \Cref{lem: infsup cont} to construct $\tilde x$ such that $\langle \cA \tilde x, x \rangle \gtrsim \| x \|_X > 0$.

    Combining \eqref{eq: lower bound A*} with \Cref{lem: continuity cont,lem: infsup cont}, we invoke the Babu\v{s}ka-Lax-Milgram theorem to conclude that a unique solution exists that satisfies \eqref{eq: bound on solution}
\end{proof}

\begin{corollary}
    The elasticity problem \eqref{eqs: weak elasticity cont} admits a unique solution $(r, u) \in R \times U$ that is bounded in the norm
    \begin{align}
        \| (r, u) \|_{R \times U}^2 := &\
        \mu^{-1} (\| r \|^2 + \| \curl r \|^2)
        + \mu \| u \|^2 + (2\mu + \lambda)\| \div u \|^2.
    \end{align}
\end{corollary}
\begin{proof}
    Let $\alpha = 0$, then the system \eqref{eq: Biot weak cont} decouples into the elasticity problem \eqref{eqs: weak elasticity cont} and a Darcy flow problem. If we neglect the Darcy problem, then \Cref{thm: well-posed cont} gives us that the unique solution $x$ is bounded in the norm $\| x \|_X = \| (r, u, 0, 0) \|_X = \| (r, u) \|_{R \times U}$.
\end{proof}

\section{Conforming four-field MFE discretization}
\label{sec:discretization}

In this section, we introduce a conforming mixed finite element discretization of the four-field formulation \eqref{eq: Biot weak cont}. Let $\Omega_h$ be a shape-regular, simplicial tesselation of the domain $\Omega$. We define the discrete space $X_h := R_h \times U_h \times Q_h \times P_h$ such that the following assumptions hold
\begin{enumerate}[label=A\arabic*., ref=A\arabic*]
    \item \label{ass: conforming}
        The finite element spaces are conforming, i.e. $X_h \subset X$.
    \item \label{ass: infsup Poisson}
        The pairs $Q_h \times P_h$ and $U_h \times P_h$ satisfy
        \begin{align*}
            \Ran(\div, Q_h) = \Ran(\div, U_h) = P_h
        \end{align*}
        and are inf-sup stable for the mixed formulation of the Poisson equation, i.e.
        \begin{align*}
            \sup_{\tilde q_h \in Q_h} \frac{\langle \div \tilde q_h, p_h \rangle_\Omega}{\| \tilde q_h \| + \| \div \tilde q_h \|} &\gtrsim \| p_h \|, &
            \sup_{\tilde u_h \in U_h} \frac{\langle \div \tilde u_h, p_h \rangle_\Omega}{\| \tilde u_h \| + \| \div \tilde u_h \|} &\gtrsim \| p_h \|, &
            \forall p_h &\in P_h.
        \end{align*}
    \item \label{ass: infsup curl}
        The pair $R_h \times U_h$ satisfies $\Ran(\curl, R_h) = \Ker(\div, U_h)$ and
        \begin{align*}
            \sup_{\tilde r_h \in R_h} \frac{\langle \curl \tilde r_h, u_h \rangle_\Omega}{\| \tilde r_h \| + \| \curl \tilde r_h \|} &\gtrsim \| u_h \|, &
            \forall u_h &\in U_h.
        \end{align*}
\end{enumerate}

\begin{remark}
    Assumptions \ref{ass: infsup Poisson} and \ref{ass: infsup curl} can be relaxed to general inf-sup stable pairs of finite elements that do not satisfy $\Ran(\div, Q_h) \subseteq P_h$, $\Ran(\div, U_h) \subseteq P_h$, and $\Ran(\curl, R_h) \subseteq U_h$. However, in order to ease the upcoming analysis, we consider these stronger assumptions.
\end{remark}
\begin{remark}
    These assumptions are different from the Stokes-Biot stability conditions introduced in \cite[Def. 3.1]{rodrigo2018new} since we do note require a Stokes-stable pair $U_h \times P_h$. On the other hand, we need the additional space $R_h$ to capture the solid rotations.
\end{remark}

We focus on two families of discretizations that satisfy these assumptions. For given polynomial degree $k \ge 0$, the first of these families is given by
\begin{align} \label{eq: def X_h1}
    X_h^{(1)} := \left(\mathbb{N}_k^{(1)} \times \mathbb{RT}_k \times \mathbb{RT}_k \times \mathbb{P}_k \right) \cap X.
\end{align}
In particular, the rotation variable is discretized using N\'ed\'elec elements of the first kind \cite{nedelec1980mixed}. Here, $k$ denotes the polynomial degree of the tangential traces on mesh edges. The displacement and fluid flux are both discretized using Raviart-Thomas elements of order $k$, which implies that the normal traces of these functions on mesh faces are of polynomial degree $k$. Finally, the pressure variable is sought in the space of discontinuous, elementwise polynomials of degree $k$. The intersection with $X$ ensures that the essential boundary conditions are respected.

The second family we consider is defined as:
\begin{align} \label{eq: def X_h2}
    X_h^{(2)} := \left(\mathbb{N}_{k + 1}^{(2)} \times \mathbb{RT}_k \times \mathbb{BDM}_{k + 1} \times \mathbb{P}_k \right) \cap X.
\end{align}
In this case, the N\'ed\'elec elements of the second kind are used to discretize the rotation variable. The notation $\mathbb{N}_{k + 1}^{(2)}$ implies that the basis functions of $R_h$ are given by polynomials of degree $k + 1$ on the mesh edges. Similarly, the fluid flux is here given by the Brezzi-Douglas-Marini space $\mathbb{BDM}_{k + 1}$ which has normal traces on mesh faces of polynomial degree $k + 1$.

In 2D, {since rotation is a scalar,} the two families of discrete spaces employ the continuous Lagrange elements of order $k + 1$ for the rotation space $R_h$, denoted by $\mathbb{L}_{k + 1}$.

With the discrete space $X_h \subset X$ defined, we are ready to formulate the \emph{four-field mixed finite element method (4F-MFEM)}:
find $x_h \in X_h$ such that
\begin{align} \label{eq: Biot weak disc}
    \langle \cA x_h, \tilde x_h \rangle &= \langle f, \tilde x_h \rangle,
    & \forall \tilde x_h &\in X_h.
\end{align}

The two main results are presented next, namely the stability of the 4F-MFEM in \Cref{sub:stability} and its convergence in \Cref{sub:convergence}.

\subsection{Stability}
\label{sub:stability}

The analysis of \eqref{eq: Biot weak disc} follows the same steps as in \Cref{sec:analysis}.
First, the continuity bound $\langle A x_h, \tilde x_h \rangle \lesssim \| x_h \|_X \| \tilde x_h \|_X, \ \forall x_h, \tilde x_h \in X_h$ is immediate by \Cref{lem: continuity cont} and the conformity \ref{ass: conforming}. The inf-sup condition is considered in the following lemma.

\begin{lemma}[Inf-sup] \label{lem: inf-sup disc}
    If $X_h$ satisfies \ref{ass: conforming}--\ref{ass: infsup curl}, then the following bound holds
    \begin{align}
        \sup_{\tilde x_h \in X_h} \frac{\langle \cA x_h, \tilde x_h \rangle}{\| \tilde x_h \|_X}
        &\gtrsim \| x_h \|_X,
        & \forall x_h &\in X_h.
    \end{align}
\end{lemma}
\begin{proof}
    First, we define the discrete analogue of the energy norm \eqref{eq: enorm} as
    \begin{align} \label{eq: enorm h}
        \enorm{x_h}_h^2 := &\
        \mu^{-1} (\| r_h \|^2 + \| \curl r_h \|^2)
        + \mu \| \Pi_h u_h \|^2 + (2\mu + \lambda)\| \div u_h \|^2 \nonumber \\
        &+ K^{-1} \| q_h \|^2
        + \frac1{\eta + c_0} \| \div (\alpha u_h + \delta q_h) \|^2
        + (\eta + c_0) \| p_h \|^2.
    \end{align}
    with $\Pi_h$ the $L^2$-projection on $\Ran(\curl, R_h)$. The equivalence $\| x_h \|_X \eqsim \enorm{x_h}_h$ holds for all $x_h \in X_h$ by the same arguments as in \Cref{lem: norm equiv}, using \ref{ass: conforming} and $\Ker(\div, U_h) = \Ran(\curl, R_h)$ from assumption \ref{ass: infsup curl}.

    Next, we follow the proof of \Cref{lem: infsup cont} for a discrete function $x_h \in X_h$ and $\Pi$ replaced by $\Pi_h$. Each of the test functions $\tilde x_{h,i} \in X_h$ can be created in analogy to $\tilde x_i$ due to assumptions \ref{ass: infsup Poisson} and \ref{ass: infsup curl}. Thus, defining $\tilde x_h := \sum_i \tilde x_{h,i}$, we have
    $\langle \cA x_h, \tilde x_h \rangle \gtrsim \enorm{x_h}_h^2$ and $\enorm{\tilde x_h}_h \lesssim \enorm{x_h}_h$. The equivalence of norms then provides the result.
\end{proof}

The stability of the mixed finite element method, which forms the main result of this section, now follows by the same arguments as in \Cref{thm: well-posed cont}.

\begin{theorem}[Stability] \label{thm: stability}
    If assumptions \ref{ass: conforming}--\ref{ass: infsup curl} are satisfied, then the discrete problem \eqref{eq: Biot weak disc} admits a unique solution $x_h \in X_h$ that satisfies
    \begin{align} \label{eq: bound on discr solution}
        \| x_h \|_X \lesssim \| f \|_{X_h'} := \sup_{\tilde x_h \in X_h} \frac{\langle f, \tilde x_h \rangle}{\| \tilde x_h \|_X}.
    \end{align}
\end{theorem}

\subsection{Convergence}
\label{sub:convergence}

Let $\Pi_R$, $\Pi_U$, $\Pi_Q$, and $\Pi_P$ be the canonical interpolation operators onto the respective finite element spaces, defined for sufficiently regular $(r, u, q, p) \in X$. These operators have the following approximation properties for the finite element families $X_h^{(1)}$ and $X_h^{(2)}$ from \eqref{eq: def X_h1} and \eqref{eq: def X_h2}:
\begin{align*}
    \| (I - \Pi_R) r \| &\lesssim h^{\bar k + 1} \| r \|_{\bar k + 1}, &
    \| \curl(I - \Pi_R) r \| &\lesssim h^{k + 1} \| \curl r \|_{k + 1}, \\
    \| (I - \Pi_U) u \| &\lesssim h^{k + 1} \| u \|_{k + 1}, &
    \| \div(I - \Pi_U) u \| &\lesssim h^{k + 1} \| \div u \|_{k + 1}, \\
    \| (I - \Pi_Q) q \| &\lesssim h^{\bar k + 1} \| q \|_{\bar k + 1}, &
    \| \div(I - \Pi_Q) q \| &\lesssim h^{k + 1} \| \div q \|_{k + 1}, \\
    \| (I - \Pi_P) p \| &\lesssim h^{k + 1} \| p \|_{k + 1}.
\end{align*}
Here, $\| \cdot \|_{k}$ denotes the $H^k(\Omega)$-norm and we have $\bar k = k$ for $X_h := X_h^{(1)}$ and $\bar k = k + 1$ for $X_h := X_h^{(2)}$. Let the composite interpolant $\Pi_X$ be defined for sufficiently regular elements of $X$ such that
\begin{align}
    \Pi_X(r, u, q, p) :=
    (\Pi_R r, \Pi_U u, \Pi_Q q, \Pi_P p).
\end{align}

\begin{theorem}[Error estimate] \label{thm: Error estimate}
    If $x$, the solution to \eqref{eq: Biot weak cont}, is sufficiently regular, and $X_h$ is chosen as $X_h^{(1)}$ or $X_h^{(2)}$ with polynomial degree $k \ge 0$, then the solution $x_h \in X_h$ to \eqref{eq: Biot weak disc} converges as
    \begin{align}
        \| x_h - x \|_X \le C h^{k + 1} \Big(&
        h^{\bar k - k}\| r \|_{\bar k + 1} + \| \curl r \|_{k + 1}
        + \| u \|_{k + 1} + \| \div u \|_{k + 1} \nonumber \\
        &+ h^{\bar k - k}\| q \|_{\bar k + 1}
        + \| \div q \|_{k + 1}
        + \| p \|_{k + 1} \Big),
    \end{align}
    with $C \ge 0$ possibly depending on the material parameters.
\end{theorem}
\begin{proof}
    As shown in the proof of \Cref{lem: inf-sup disc}, for each $y_h \in X_h$, a $\tilde y_h$ exists such that $\langle \cA y_h, \tilde y_h \rangle \gtrsim \| y_h \|_X^2$ and $\| \tilde y_h \|_X \lesssim \| y_h \|_X$. Consider $y_h = x_h - \Pi_X x \in X_h$ and let us use the properties of $\tilde y_h$ with the consistency $X_h \subset X$ and continuity of $\cA$ from \Cref{lem: continuity cont} to derive
    \begin{align*}
        \| x_h - \Pi_X x \|_X^2
        \lesssim \langle \cA (x_h - \Pi_X x), \tilde y_h \rangle
        &= \langle \cA (x - \Pi_X x), \tilde y_h \rangle \\
        &\lesssim \| (I - \Pi_X) x \|_X \| x_h - \Pi_X x \|_X.
    \end{align*}
    A triangle inequality now gives us
    \begin{align*}
        \| x_h - x \|_X
        &\le
        \| x_h - \Pi_X x \|_X
        + \| (I - \Pi_X) x \|_X
        \lesssim \| (I - \Pi_X) x \|_X,
    \end{align*}
    and the approximation properties of $\Pi_X$ conclude the proof.
\end{proof}

\section{A multipoint rotation-flux MFE method}
\label{sec:a_multipoint_rotation_flux_mfe_method}

We continue by considering the lowest order instance of the finite element family of the second kind, i.e. $X_h^{(2)}$ from \eqref{eq: def X_h2} with $k = 0$. In 3D, this space contains two degrees of freedom per edge for the rotation $r \in \mathbb{N}_1^{(2)}$ and three degrees of freedom per face for the flux $q \in \mathbb{BDM}_1$.
By choosing an appropriate quadrature rule for the inner products in $R_h$ and $Q_h$, we can localize the mass matrix around the vertices. In turn, the variables $(r,q)$ can be eliminated through static condensation and we obtain a multipoint mixed finite element method for the Biot system with $(u,p) \in \mathbb{RT}_0 \times \mathbb{P}_0$ in \Cref{sub:static_condensation}. The stability and convergence of this method are shown in \Cref{sub:analysis_reducible}.

\subsection{Static condensation}
\label{sub:static_condensation}

Following \cite{lee2018local}, we consider a specific quadrature rule for functions $\varphi, \phi$ on a simplicial mesh:
\begin{align} \label{eq: quadrature rule}
    \langle \phi, \varphi \rangle_h :=
    \sum_{\omega \in \Omega_h} \frac{|\omega|}{n + 1} \sum_{\bm x \in \mathcal{V}(\omega)} \phi_\omega(\bm x) \cdot \varphi_\omega(\bm x)
\end{align}
in which $\phi_\omega$ is the restriction of $\phi$ on an element $\omega \in \Omega_h$ and $\mathcal{V}(\omega)$ is the set of its vertices. The dot product herein reduces to the standard product for scalar $\phi, \varphi$.

We emphasize that $\langle \phi, \varphi \rangle_h$ is non-zero if and only if the basis functions $\phi, \varphi$ are non-zero at the same vertex of a mesh element. It has two more advantageous properties, which we highlight in the following lemma.

\begin{lemma}[{\cite[Thm. 4.1]{lee2018local}}] \label{lem: properties quad}
The norm $\| \cdot \|_h$ induced by the inner product \eqref{eq: quadrature rule} is equivalent to the $L^2$-norm on $R_h$ and $Q_h$:
\begin{align}
    \| \phi \|_h &\eqsim \| \phi \|, &
    \| \varphi \|_h &\eqsim \| \varphi \|, &
    \forall (\phi, \varphi) &\in R_h \times Q_h.
\end{align}
Moreover, the quadrature rule is exact if the test function is elementwise constant:
\begin{subequations}
\begin{align}
    \langle r_h, \mathring{r}_h \rangle_h
    &= \langle r_h, \mathring{r}_h \rangle_\Omega, &
    \forall r_h \in
    R_h,\ \mathring{r}_h &\in \mathring{R}_h := \mathbb{P}_0^{k_n}, \\
    \langle q_h, \mathring{q}_h \rangle_h
    &= \langle q_h, \mathring{q}_h \rangle_\Omega, &
    \forall q_h \in
    Q_h,\ \mathring{q}_h &\in \mathring{Q}_h := \mathbb{P}_0^n.
\end{align}
\end{subequations}
\end{lemma}

We will apply this quadrature rule on the $L^2$-inner products on the two spaces $R_h \times Q_h$. In particular, we substitute $\langle \cdot, \cdot \rangle_h$ in the definition of $A$ to obtain the discrete operator $A_h: X_h \to X_h'$:
\begin{align}
    \langle A_h x_h, \tilde x_h \rangle :=
    \mu^{-1} \langle r_h, \tilde r_h \rangle_h
    + (2\mu + \lambda) \langle \div u_h, \div \tilde u_h \rangle_\Omega
    + \langle K^{-1} q_h, \tilde q_h \rangle_h
    + c_0 \langle p_h, \tilde p_h \rangle_\Omega.
\end{align}

In turn, we define $\cA_h := A_h + B - B^*$, which leads us to the following problem: find $\hat x_h \in X_h$ such that
\begin{align} \label{eq: reducible problem}
    \langle \cA_h \hat x_h, \tilde x_h \rangle
    &= \langle f, \tilde x_h \rangle, &
    \forall \tilde x_h &\in X_h.
\end{align}

Next, we aim to eliminate the variables $\hat r_h$ and $\hat q_h$ to obtain a multipoint mixed finite element method. For that, let us consider the matrix representation of \eqref{eq: reducible problem}.
Let $\mathsf{M_u}$ and $\mathsf{M_p}$ be the mass matrices on the spaces $U_h$ and $P_h$, respectively. Moreover, let $\mathsf{M_{r, h}}$ and $\mathsf{M_{q, h}}$ be the matrices corresponding to the quadrature rule, applied to the basis functions of $R_h$ and $Q_h$. Let $\mathsf{B_r}$, $\mathsf{B_u}$, and $\mathsf{B_q}$ be the representations of the curl on $R_h$, the divergence on $U_h$, and the divergence on $Q_h$, respectively. Finally, to shorten notation, we let $\mathsf{\hat{B}_r} = \mathsf{M_u B_r}$, $\mathsf{\hat{B}_u} = \mathsf{M_p B_u}$, and $\mathsf{\hat{B}_q} = \mathsf{M_p B_q}$ denote the action of these differential operators in their respective range spaces.

Since $\mathsf{M_{r, h}}$ and $\mathsf{M_{q, h}}$ are easily invertible, the variables $\hat r_h$ and $\hat q_h$ can now be eliminated by taking a Schur complement. We arrive at the algebraic formulation of the \emph{multipoint rotation-flux mixed finite element method (MR-MFEM)}: find $(\hat u, \hat p) \in U_h \times P_h$ such that
\begin{align} \label{eq: reduced problem algebraic}
    \begin{bmatrix}
        (2\mu + \lambda) \mathsf{B_u^T M_u B_u} + \mu \mathsf{ \hat{B}_r M_{r, h}^{-1} \hat{B}_r^T } & - \alpha \mathsf{\hat{B}_u^T} \\
        \alpha \mathsf{\hat{B}_u} & c_0 \mathsf{M_p} + \delta^2 K \mathsf{ \hat{B}_q M_{q, h}^{-1} \hat{B}_q^T }
    \end{bmatrix}
    \begin{bmatrix}
        \mathsf{\hat u} \\ \mathsf{\hat p}
    \end{bmatrix}
    =
    \begin{bmatrix}
        \mathsf{f_u} - \mu \mathsf{ \hat{B}_r M_{r, h}^{-1} f_r } \\
        \mathsf{f_p} - \delta K \mathsf{ \hat{B}_q M_{q, h}^{-1} f_q }
    \end{bmatrix},
\end{align}
with $\mathsf{\hat u}$,  $\mathsf{\hat p}$ the vector representations of $\hat u_h$ and $\hat p_h$, respectively, and $\mathsf{f} = [\mathsf{f_r}, \mathsf{f_q}, \mathsf{f_u}, \mathsf{f_p}]^T$ representing the right-hand side $f \in X_h'$ in \eqref{eq: Biot weak disc}.

We will refer to \eqref{eq: reducible problem} as the \emph{reducible problem} and the equivalent \eqref{eq: reduced problem algebraic} as the \emph{reduced problem}.
We remark that \eqref{eq: reduced problem algebraic} is a discretization of the following system, in which we recognize the $(1,1)$-block as the weighted vector Laplacian from \eqref{eq: momentum balance}:
\begin{align} \label{eq: reduced problem}
    \begin{bmatrix}
        - \grad (2\mu + \lambda) \div + \curl \mu \curl & \grad \alpha \\
        \alpha \div & c_0 - \div \delta^2 K \grad
    \end{bmatrix}
    \begin{bmatrix}
        \hat u \\ \hat p
    \end{bmatrix}.
\end{align}

\begin{remark}
    The more general case in which the conductivity $K$ is an elementwise constant, full tensor can be handled by defining a new lumped mass matrix $\mathsf{M_{q, h}^K}$ such that $\mathsf{q^T M_{q, h}^K \tilde q} = \langle K^{-1} q, \tilde q \rangle_h$ for all $q, \tilde q \in Q_h$. However, we restrict our analysis to constant, scalar $K$ and will only consider the more general case in the numerical examples of \Cref{sec: Numerical results}.
\end{remark}

\subsection{Analysis}
\label{sub:analysis_reducible}

We devote this subsection to the theoretical results concerning the MR-MFEM \eqref{eq: reduced problem algebraic}. After presenting general findings, we present specific results for the two-dimensional case in \Cref{ssub: two-dimensional case}.

\begin{lemma}[Well-posedness] \label{lem: well-posed MR-MFEM}
    The operator $\cA_h$ satisfies
    \begin{align}
        \| x_h \|_X
        &\eqsim
        \sup_{\tilde x_h \in X_h} \frac{\langle \cA_h x_h, \tilde x_h \rangle}{\| \tilde x_h \|_X},
        & \forall x_h &\in X_h.
    \end{align}
    In turn, the reducible problem \eqref{eq: reducible problem} admits a unique and bounded solution $\hat x_h \in X_h$.
\end{lemma}
\begin{proof}
    The same arguments as in \Cref{thm: stability} are followed, using the equivalence $\| \cdot \| \eqsim \| \cdot \|_h$ when necessary.
\end{proof}

As in \Cref{lem: properties quad}, we let $\mathring{R}_h := \mathbb{P}_0^{k_n}$ with $k_n := \left(\begin{smallmatrix} n \\ 2 \end{smallmatrix}\right)$ and $\mathring{Q}_h := \mathbb{P}_0^n$ be the spaces containing elementwise constant (vector) functions. Additionally, let $\mathring{\Pi}_R$ and $\mathring{\Pi}_Q$ be their respective $L^2$ projections. These have the following approximation properties for sufficiently regular $(r, q) \in R \times Q$:
\begin{align} \label{eq: pw constant approx}
    \| (I - \mathring{\Pi}_R) r \| &\lesssim h \| r \|_1, &
    \| (I - \mathring{\Pi}_Q) q \| &\lesssim h \| q \|_1.
\end{align}

\begin{lemma}[Convergence] \label{lem: Convergence MR-MFEM}
    The solution $\hat x_h$ to \eqref{eq: reducible problem} converges linearly to $x$, the solution to \eqref{eq: Biot weak cont}. I.e. a constant $C$ exists, depending on the physical parameters and the regularity of $x$, such that
    \begin{align}
        \| \hat x_h - x \|_X \le Ch.
    \end{align}
\end{lemma}
\begin{proof}
    We follow \cite[Thm. 3.2]{lee2018local} by first showing that the solutions $\hat x_h$ and $x_h$ converge linearly to each other and then using \Cref{thm: Error estimate} to obtain the result. We start by considering the norm of the difference and use \Cref{lem: well-posed MR-MFEM} with the fact that $\mathcal{A} x_h = \mathcal{A}_h \hat x_h$ to derive:
    \begin{align*}
        \| \hat x_h - x_h \|_X
        \eqsim
        \sup_{\tilde x_h \in X_h} \frac{\langle \cA_h(\hat x_h - x_h), \tilde x_h \rangle}{\| \tilde x_h \|_X}
        &=
        \sup_{\tilde x_h \in X_h} \frac{\langle (\cA - \cA_h) x_h, \tilde x_h \rangle}{\| \tilde x_h \|_X} \\
        &=
        \sup_{\tilde x_h \in X_h} \frac{\langle (A - A_h) x_h, \tilde x_h \rangle}{\| \tilde x_h \|_X}
    \end{align*}

    We continue by bounding the numerator, which consists of the following terms:
    \begin{align*}
        \langle (A - A_h) x_h, \tilde x_h \rangle
        &= \mu^{-1} (\langle r_h, \tilde r_h \rangle_\Omega
                - \langle r_h, \tilde r_h \rangle_h)
        + K^{-1} (\langle q_h, \tilde q_h \rangle_\Omega
                - \langle q_h, \tilde q_h \rangle_h)
    \end{align*}
    For the first term, we derive the upper bound
    \begin{align*}
        \mu^{-1} (\langle r_h, \tilde r_h \rangle_\Omega
                    - \langle r_h, \tilde r_h \rangle_h)
        &= \mu^{-1} (\langle r_h - \mathring{\Pi}_R r, \tilde r_h \rangle_\Omega
                    + \langle \mathring{\Pi}_R r - r_h, \tilde r_h \rangle_h) \\
        &\le \mu^{-1} (\| r_h - \mathring{\Pi}_R r \| \| \tilde r_h \| + \| \mathring{\Pi}_R r - r_h \|_h \| \tilde r_h \|_h) \\
        &\eqsim \mu^{-1} \| r_h - \mathring{\Pi}_R r \| \| \tilde r_h \|
    \end{align*}
    The second term is bounded analogously and we obtain
    \begin{align*}
        \langle (A - A_h) x_h, &\tilde x_h \rangle
        \lesssim \mu^{-1} \| r_h - \mathring{\Pi}_R r \| \| \tilde r_h \|
        + K^{-1} \| q_h - \mathring{\Pi}_Q q \| \| \tilde q_h \| \\
        &\lesssim (\mu^{-1} \| r_h - \mathring{\Pi}_R r \|^2 + K^{-1} \| q_h - \mathring{\Pi}_Q q \|^2)^{\frac12}
        (\mu^{-1} \| \tilde r_h \|^2 + K^{-1} \| \tilde q_h \|^2)^{\frac12} \\
        &\le (\mu^{-1} \| r_h - \mathring{\Pi}_R r \|^2 + K^{-1} \| q_h - \mathring{\Pi}_Q q \|^2)^{\frac12}
        \| \tilde x_h \|_X
    \end{align*}

    Combining with the above, we have thus derived
    \begin{align*}
        \| \hat x_h - x_h \|_X &\lesssim (\mu^{-1} \| r_h - \mathring{\Pi}_R r \|^2 + K^{-1} \| q_h - \mathring{\Pi}_Q q \|^2)^{\frac12} \\
        &\lesssim \sqrt{\mu^{-1}} (\| r_h - r \| + \| r - \mathring{\Pi}_R r \|)
        + \sqrt{K^{-1}} (\| q_h - q \| + \| q - \mathring{\Pi}_Q q \|) \\
        &\lesssim \| x_h - x \|_X
        + \sqrt{\mu^{-1}} \| r - \mathring{\Pi}_R r \|
        + \sqrt{K^{-1}} \| q - \mathring{\Pi}_Q q \|.
    \end{align*}
    Finally, a triangle inequality, property \eqref{eq: pw constant approx}, and \Cref{thm: Error estimate} with $k = 0$ give us
    \begin{align*}
        \| \hat x_h - x \|_X
        &\le \| \hat x_h - x_h \|_X + \| x_h - x \|_X
        \le C h.
    \end{align*}
\end{proof}

The introduction of the quadrature rule leaves {components} of the solution unchanged. We present these invariants formally in the following lemma and subsequent corollaries. The proofs are analogous to \cite{boon2022multipoint}, but are included here for the sake of completeness.

\begin{lemma} \label{lem: invariant curl}
    The {application of the} quadrature rule does not affect the curl of the rotation, i.e. $\curl \hat r_h = \curl r_h$.
\end{lemma}
\begin{proof}
    Let $u_r = \curl (\hat r_h - r_h)$ and consider the test function $\tilde x = (0, u_r, 0, 0)$. Using the fact that $\div u_r = 0$, we derive
    \begin{align*}
        0 &= \langle \cA_h \hat x_h - \cA x_h, \tilde x \rangle \\
        &= \langle \curl (\hat r_h - r_h), u_r \rangle_\Omega
        + (2\mu + \lambda) \langle \div (\hat u_h - u_h), \div u_r \rangle_\Omega
        - \langle \hat p_h - p_h, \div \alpha u_r \rangle_\Omega \\
        &= \| \curl (\hat r_h - r_h) \|^2.
    \end{align*}
\end{proof}

\begin{corollary} \label{cor: invariant div u}
    In the decoupled case of $\alpha = 0$, the volumetric strain is not affected by the quadrature rule. I.e. $\div \hat u_h = \div u_h$.
\end{corollary}
\begin{proof}
    Let us consider $\tilde x = (0, \hat u_h - u_h, 0, 0)$. Using $\curl (\hat r_h - r_h) = 0$ from \Cref{lem: invariant curl} and $\alpha = 0$, we derive:
    \begin{align*}
        0 &= \langle \cA_h \hat x_h - \cA x_h, \tilde x \rangle
        = (2\mu + \lambda) \| \div (\hat u_h - u_h) \|^2.
    \end{align*}
\end{proof}


\subsubsection{The two-dimensional case}
\label{ssub: two-dimensional case}

In the special case with $n = 2$, we obtain stronger results concerning the rotation variable, which we present in the following corollary and lemma, respectively.

\begin{corollary} \label{cor: invariant rot 2D}
    In 2D, the rotation variable is entirely unaffected by the quadrature rule: $\hat r_h = r_h$.
\end{corollary}
\begin{proof}
    In 2D, the curl is given by the rotated gradient $\curl u = (- \partial_2 u, \partial_1 u)$. In turn, \Cref{lem: invariant curl} implies that $\hat r_h - r_h$ is a constant. For $\tilde x := (\mu(\hat r_h - r_h), 0, 0, 0)$, we derive using $\mathbb{R} \subseteq \mathring{R}_h$ and \Cref{lem: properties quad}:
    \begin{align*}
        0 = \langle \cA_h \hat x_h - \cA x_h, \tilde x \rangle
        &= \langle \hat r_h, \hat r_h - r_h \rangle_h - \langle r_h, \hat r_h - r_h \rangle_\Omega \\
        &= \langle \hat r_h, \hat r_h - r_h \rangle_\Omega - \langle r_h, \hat r_h - r_h \rangle_\Omega
        = \| \hat r_h - r_h \|^2.
    \end{align*}
\end{proof}

\begin{lemma}[Improved estimate] \label{lem: improved estimate}
    If $n = 2$, then we obtain second order convergence in the rotation variable:
    \begin{align}
        \| \hat r_h - r \| =
        \| r_h - r \| \lesssim
        h^2 \| r \|_2
    \end{align}
\end{lemma}
\begin{proof}
    The equality was shown in \Cref{cor: invariant rot 2D}. For the second, we introduce the projection $\mathcal{P}_R: R \to R_h$ obtained by solving the following problem: find $\mathcal{P}_R r \in R_h$ such that
    \begin{align} \label{eq: orthogonality P_R}
        \langle \curl \mathcal{P}_R r, \curl \tilde r_h \rangle_\Omega
        &= \langle \curl r, \curl \tilde r_h \rangle_\Omega, &
        \forall \tilde r_h &\in R_h.
    \end{align}
    In case $\partial_r \Omega = \emptyset$, we set $\langle \mathcal{P}_R, 1 \rangle_\Omega = \langle r, 1 \rangle_\Omega$ to ensure uniqueness. Since $\curl$ is a rotated gradient in 2D, $\mathcal{P}_R r$ is an $\mathbb{L}_1$ approximation of the solution to a Laplace problem, which gives us the approximation property
    \begin{align} \label{eq: approximation P_R}
    \| (I - \mathcal{P}_R) r \| \lesssim h^2 \| r \|_2.
    \end{align}

    Let the test function $\tilde x := (r_h - \mathcal{P}_R r, \Pi_h(u_h - u),0,0)$ with $\Pi_h$ the $L^2$ projection onto $\Ran(\curl, R_h) \subseteq Q_h$, as in the proof of \Cref{lem: inf-sup disc}. We now use $\div \Pi_h = 0$ and the orthogonality property \eqref{eq: orthogonality P_R} to derive
    \begin{align*}
        0 =
        \langle \mathcal{A} (x_h - x), \tilde x \rangle
        &= \mu^{-1} \langle  r_h - r, r_h - \mathcal{P}_R r \rangle_\Omega
        - \langle u_h - u, \curl (r_h - \mathcal{P}_R r) \rangle_\Omega \\
        &\quad + \langle \Pi_h (u_h - u), \curl (r_h - r) \rangle_\Omega \\
        &= \mu^{-1} \langle  r_h - r, r_h - \mathcal{P}_R r \rangle_\Omega
        - \langle \Pi_h (u_h - u), \curl (I - \mathcal{P}_R) r \rangle_\Omega \\
        &= \mu^{-1} \langle  r_h - r, r_h - \mathcal{P}_R r \rangle_\Omega \\
        &= \mu^{-1} (\| r_h - r \|^2 + \langle  r_h - r, (I - \mathcal{P}_R) r \rangle_\Omega)
    \end{align*}

    Combining this with \eqref{eq: approximation P_R}, we arrive at the following bound
    \begin{align*}
        \| r_h - r \|
        = \frac{- \langle r_h - r, (I - \mathcal{P}_R) r \rangle_\Omega}{\| r_h - r \|}
        \le \| (I - \mathcal{P}_R) r \|
        \lesssim h^2 \| r \|_2.
    \end{align*}
\end{proof}

\section{Parameter-robust preconditioning for 4F-MFEM}
\label{sec: Parameter-robust preconditioning}

We follow the preconditioning framework \cite{mardal2011preconditioning}, which uses the Riesz representation operator as the canonical preconditioner.
In particular, let $\langle \cdot, \cdot \rangle_X$ be the inner product on $X$ that induces the norm $\| \cdot \|_X$ from \eqref{eq: norm X}. Then we define the preconditioner $\mathcal{P}: X_h' \to X_h$ as the operator that satisfies
\begin{align}
    \langle \mathcal{P} f, \tilde x \rangle_X &= \langle f, \tilde x \rangle,
    & \forall (f, \tilde x) &\in X_h' \times X_h.
\end{align}

Let $\mathcal{L}(X_h, X_h)$ denote the space of linear maps $X_h \to X_h$ and let $\| \cdot \|_{\mathcal{L}(X_h, X_h)}$ be its norm. From \Cref{lem: continuity cont}, we have $\| \mathcal{PA} \|_{\mathcal{L}(X_h, X_h)} \lesssim 1$ and, additionally, \Cref{lem: inf-sup disc} implies that $\| (\mathcal{PA})^{-1} \|_{\mathcal{L}(X_h, X_h)}^{-1} \gtrsim 1$. In turn, the condition number of the preconditioned system satisfies
\begin{align}
    \kappa(\mathcal{PA}) =
    \| \mathcal{PA} \|_{\mathcal{L}(X_h, X_h)} \| (\mathcal{PA})^{-1} \|_{\mathcal{L}(X_h, X_h)}
    \lesssim 1.
\end{align}

Since our analysis is based on bounds with constants that are independent of material parameters, this condition number is bounded from above for all admissible parameters. In turn, the operator $\mathcal{P}$ is a parameter-robust preconditioner for 4F-MFEM.

Due to the definition of the norm $\| \cdot \|_X$, the operator $\mathcal{P}$ has a block-diagonal structure
with the blocks defined according to their inverses:
\begin{align*}
    \langle \mathcal{P}_r^{-1} r, \tilde r \rangle &:=
        \mu^{-1} (\langle r, \tilde r \rangle_\Omega + \langle \curl r, \curl \tilde r \rangle_\Omega), \\
    \langle \mathcal{P}_u^{-1} u, \tilde u \rangle &:=
        \mu \langle u, \tilde u \rangle_\Omega + (2\mu + \lambda)\langle \div u, \div \tilde u \rangle_\Omega, \\
    \langle \mathcal{P}_q^{-1} q, \tilde q \rangle &:=
        K^{-1} \langle q, \tilde q \rangle_\Omega + \frac{\delta^2}{\eta + c_0} \langle \div q, \div \tilde q \rangle_\Omega, &
    \langle \mathcal{P}_p^{-1} p, \tilde p \rangle &:=
        (\eta + c_0) \langle p, \tilde p \rangle_\Omega.
\end{align*}

We emphasize that this preconditioner requires solving independent systems for the variables $r$, $u$, $q$, and $p$. In order to generate a preconditioner that is scalable for larger systems, these solves can be replaced by spectrally equivalent operators \cite{mardal2011preconditioning}. However, such extensions are beyond the scope of this work.

\section{Numerical results}
\label{sec: Numerical results}

In this section, we present numerical experiments to show the performance of the proposed schemes.
We first perform convergence studies in \Cref{subsec:error} to verify the results from \Cref{sec:discretization} and \Cref{sec:a_multipoint_rotation_flux_mfe_method}. Subsequently, \Cref{sub:parameter_robust_preconditioning} shows
the robustness of the preconditioner introduced in
\Cref{sec: Parameter-robust preconditioning}. Finally, \Cref{subsec:mandel} presents the
approximations for the Mandel problem computed by the proposed four-field and multipoint schemes.

We focus on the lowest order instances of the families of finite elements introduced in \Cref{sec:discretization}. We refer to the first as the two-field mixed finite element method for elasticity (\emph{2F-MFEM}), respectively the four-field MFEM for poroelasticity (\emph{4F-MFEM}). For the second kind, we apply the quadrature rule from \Cref{sec:a_multipoint_rotation_flux_mfe_method}, and we use the acronym \emph{MR-MFEM} to refer to the resulting multipoint rotation(-flux) mmixed finite element method.

All the numerical results are obtained with the libraries PorePy
\cite{keilegavlen2021porepy} and PyGeoN \cite{pygeon}. The scripts of all the test are publicly
available at \url{https://github.com/compgeo-mox/rotation_based_biot}.

\subsection{Convergence study}\label{subsec:error}

In this section we evaluate the performance of the method by considering the numerical errors. 
In particular, we consider problems \eqref{eq:mechanics} and
\eqref{eqs: Biot strong cont} for $n=2$ and $n=3$, with the computational domain
given by the unit cube $\Omega=(0, 1)^n$. For
simplicity, all material parameters are set to 1 and we assume homogeneous essential conditions on $\partial \Omega$.

\subsubsection{Linear elasticity}
\label{ssub: numerics Linear elasticity}

We consider the following
exact solutions for $n=2$ and 3 of Problem \eqref{eq:mechanics}, respectively:
\begin{subequations} \label{eq:sol1}
\begin{align}
    u(x, y) &=x^2y^2(1 - x)^2(1 - y)^2
    [4, -1]^T, \\
    u(x, y, z) &=x^2y^2z^2(1 - x)^2(1 - y)^2(1 - z)^2
    [4, -1, 2]^T,
\end{align}
\end{subequations}
and we set $r = \nabla \times u$. The source term $f_u$ is computed accordingly.

The relative $L^2$ errors against the analytical solutions are reported in
\Cref{tab:errmechdiff}.  We notice
second order convergence of $r$ for $n=2$ for both 2F-MFEM and MR-MFEM as shown in
\Cref{lem: improved estimate}. In
all the other cases, the unknowns converge linearly with respect to the mesh
size, as expected by \Cref{thm: Error estimate} and \Cref{lem: Convergence MR-MFEM}. Moroever, the number of degrees of freedom is significantly smaller for
MR-MFEM compared to the 2F-MFEM, and the errors are observed to be larger for the former than for the latter.

\begin{table}[htb]
    \centering
    \scriptsize
    \caption{Relative $L^2$ errors and convergence rates for the solutions $r$ and $u$,
    curl of the rotation $\nabla \times r$, and divergence of the displacement $\nabla \cdot u$ against the analytical solutions.
    Results for the elasticity example in \Cref{ssub: numerics Linear elasticity}.}
    \label{tab:errmechdiff}
    \setlength{\tabcolsep}{4pt}
    \begin{tabular}{cc|ccccccccc}
        & & \multicolumn{9}{|c}{2F-MFEM}\\
        & $h$
        & $N_{\text{dof}}$
        & $\text{Err}_r$
        & $\text{Rate}_r$
        & $\text{Err}_u$
        & $\text{Rate}_u$
        & $\text{Err}_{\nabla \times r}$
        & $\text{Rate}_{\nabla \times r}$
        & $\text{Err}_{\nabla \cdot u}$
        & $\text{Rate}_{\nabla \cdot u}$
        \\
        \hline
        \multirow{5}{*}{\rotatebox[origin=c]{90}{$n = 2$}}
        & 6.42e-2 & 1297   & 4.72e-3 & -    & 4.47e-2 & -    &  2.19e-2 & -    & 1.85e-1 & -   \\
        & 3.17e-2 & 4929   & 1.09e-3 & 2.07 & 2.22e-2 & 1.00 &  9.92e-3 & 1.13 & 9.28e-2 & 0.98\\
        & 1.57e-2 & 19297  & 2.47e-4 & 2.11 & 1.11e-2 & 1.00 &  3.53e-3 & 1.47 & 4.66e-2 & 0.98\\
        & 7.83e-3 & 76465  & 5.92e-5 & 2.06 & 5.52e-3 & 1.00 &  1.36e-3 & 1.37 & 2.34e-2 & 0.99\\
        & 3.91e-3 & 304473 & 1.41e-5 & 2.07 & 2.76e-3 & 1.00 &  5.00e-4 & 1.45 & 1.17e-2 & 1.00\\
        \hline
        \hline
        \multirow{5}{*}{\rotatebox[origin=c]{90}{$n = 3$}}
        & 2.34e-1 & 5900   & 1.14e-1 & -     & 1.02e-1 & -    & 1.95e-1 & -    & 6.12e-1 & - \\
        & 1.48e-1 & 22305  & 6.65e-2 & 1.18  & 5.34e-2 & 1.42 & 1.41e-1 & 0.71 & 3.69e-1 & 1.10\\
        & 1.08e-1 & 54435  & 4.26e-2 & 1.42  & 3.75e-2 & 1.13 & 9.77e-2 & 1.16 & 2.73e-1 & 0.96\\
        & 8.58e-2 & 106001 & 3.28e-2 & 1.13  & 2.88e-2 & 1.14 & 8.09e-2 & 0.81 & 2.14e-1 & 1.05\\
        & 7.09e-2 & 185943 & 2.74e-2 & 0.96  & 2.34e-2 & 1.09 & 6.56e-2 & 1.10 & 1.76e-1 & 1.03\\
        \hline
        \hline
        & & \multicolumn{9}{|c}{MR-MFEM}\\
        \hline
        \multirow{5}{*}{\rotatebox[origin=c]{90}{$n = 2$}}
        & 6.42e-2 & 956    & 4.72e-3 & -    & 4.46e-2 & -   & 2.19e-2 & -    & 1.85e-1 & -     \\
        & 3.17e-2 & 3664   & 1.09e-3 & 2.07 & 2.21e-2 & 0.99& 9.92e-3 & 1.13 & 9.28e-2 & 0.98  \\
        & 1.57e-2 & 14408  & 2.47e-4 & 2.11 & 1.11e-2 & 0.99& 3.53e-3 & 1.47 & 4.66e-2 & 0.98  \\
        & 7.83e-3 & 57220  & 5.92e-5 & 2.06 & 5.52e-3 & 1.00& 1.36e-3 & 1.37 & 2.34e-2 & 0.99  \\
        & 3.91e-3 & 228098 & 1.41e-5 & 2.07 & 2.76e-3 & 1.00& 5.00e-4 & 1.45 & 1.17e-2 & 1.00  \\
        \hline
        \hline
        \multirow{5}{*}{\rotatebox[origin=c]{90}{$n = 3$}}
        & 2.34e-1 & 3510   & 2.67e-1 & -    & 1.11e-1 & -    & 1.95e-1 & -    & 6.12e-1 & -     \\
        & 1.48e-1 & 13576  & 1.22e-1 & 1.71 & 5.48e-2 & 1.54 & 1.41e-1 & 0.71 & 3.69e-1 & 1.10  \\
        & 1.08e-1 & 33462  & 8.16e-2 & 1.29 & 3.83e-2 & 1.15 & 9.77e-2 & 1.16 & 2.73e-1 & 0.96  \\
        & 8.58e-2 & 65554  & 6.05e-2 & 1.29 & 2.92e-2 & 1.17 & 8.09e-2 & 0.81 & 2.14e-1 & 1.05  \\
        & 7.09e-2 & 115434 & 4.84e-2 & 1.17 & 2.36e-2 & 1.11 & 6.56e-2 & 1.10 & 1.76e-1 & 1.03
    \end{tabular}
\end{table}

In \Cref{tab:errmechcomparison}, we compare the solutions of the two
methods by computing the relative norms of their differences and associated convergence rates. 
We observe that the curl of the rotation $\curl r$ and the divergence of the displacement
$\div u$ are unaffected by the quadrature rule \eqref{eq: quadrature
rule}, as was shown in \Cref{lem: invariant curl}
and \Cref{cor: invariant div u}, respectively. 
Moreover, the computated rotation is identical for the two methods if $n = 2$, in agreement with \Cref{cor: invariant rot 2D}.

\begin{table}[htb]
    \centering
    \scriptsize    
    \caption{Relative differences in $L^2$ between the
    solutions obtained with 2F-MFEM and MR-MFEM for the rotation $r$,
    displacement $u$, curl of the rotation $\nabla \times r$, and divergence of the displacement $\nabla \cdot u$ of \eqref{eq:mechanics}. Results for the elasticity example in \Cref{ssub: numerics Linear elasticity}.}
    \label{tab:errmechcomparison}
    \setlength{\tabcolsep}{4pt}
    \begin{tabular}{cc|cccccc}
        & & \multicolumn{6}{|c}{2F-MFEM vs MR-MFEM}\\
        & $h$ & $\text{Err}_r$ & $\text{Rate}_r$ &
        $\text{Err}_u$ & $\text{Rate}_u$ & $\text{Err}_{\nabla \times r}$
        & $\text{Err}_{\nabla \cdot u}$
        \\
        \hline
        \multirow{5}{*}{\rotatebox[origin=c]{90}{$n = 2$}}
        & 6.42e-2 & 1.56e-14 & - & 1.21e-2 & -    & 7.18e-14 & 8.97e-15\\
        & 3.17e-2 & 5.02e-14 & - & 3.08e-3 & 1.94 & 3.44e-13 & 2.40e-14\\
        & 1.57e-2 & 1.21e-13 & - & 7.71e-4 & 1.97 & 1.62e-12 & 3.97e-14\\
        & 7.83e-3 & 3.33e-13 & - & 1.93e-4 & 1.99 & 7.89e-12 & 1.07e-13\\
        & 3.91e-3 & 1.22e-12 & - & 4.82e-5 & 2.00 & 3.65e-11 & 6.30e-13\\
        \hline
        \hline
        \multirow{5}{*}{\rotatebox[origin=c]{90}{$n = 3$}}
        & 2.34e-1 & 1.54e-1 & -    & 5.61e-2 & -    & 5.52e-14 & 6.58e-15  \\
        & 1.48e-1 & 1.01e-1 & 0.93 & 2.21e-2 & 2.04 & 3.42e-13 & 2.26e-14  \\
        & 1.08e-1 & 2.16e-2 & 1.07 & 1.16e-2 & 2.06 & 9.46e-13 & 4.94e-14  \\
        & 8.58e-2 & 5.71e-2 & 1.01 & 7.26e-3 & 2.02 & 1.39e-12 & 8.57e-14 \\
        & 7.09e-2 & 4.67e-2 & 1.06 & 4.91e-3 & 2.05 & 3.59e-12 & 1.89e-13
    \end{tabular}
\end{table}

\subsubsection{Poroelasticity}
\label{ssub: numerics Poroelasticity}

Next, we consider Problem \eqref{eqs: Biot strong cont}. Let the exact solutions for the
displacement be given by \eqref{eq:sol1}, while for the flow variables, we consider:
\begin{gather*}
    q(x, y) = [\sin(2x\pi)\sin(2y\pi),
    xy(1 - x)(1 - y)]^T,
    \quad
    p(x, y) =
    xy(1 - x)(1 - y),\\
    q(x, y, z) =
    \begin{bmatrix}
        \sin(2x\pi)\sin(2y\pi)\sin(2z\pi)\\
        xyz(1 - x)(1 - y)(1-z)\\
        y(1 - y)\sin(2x\pi)\sin(2z\pi)
    \end{bmatrix},
    \quad
    p(x, y, z) =
    xyz(1 - x)(1 - y)(1-z).
\end{gather*}
Moreover, we set $r = \nabla \times u$, and the source terms in the second and fourth
equations of \eqref{eq: Biot system} are computed accordingly. For
simplicity we have considered also a vector source term in the third equation of
\eqref{eq: Biot system}, which does not affect the previously
introduced theory.

In \Cref{tab:errbiothyb} we present the relative $L^2$ errors against the
analytical solutions of the unknowns
for both methods in 2D and 3D. We notice that in all cases the errors decay
at least with order one, which is expected by \Cref{thm: Error estimate} and \Cref{lem: Convergence MR-MFEM}. 
The
rotation in 2D is again second order convergent as observed in the previous example
and supported by \Cref{lem: improved estimate}.
We notice that in 3D, the errors for the MR-MFEM are higher than
4F-MFEM, also higher than in the previous example, probably due to the
fact that we are now performing hybridization on two variables. The inital higher order of convergence are due to the high errors obtained on the coarse grids we start with.

\Cref{tab:errbiotdifferentialmulti} contains the relative $L^2$ errors for the
differentials, i.e. the relevant curl and divergence, of the numerical solutions against their analytical counterparts.
Also in this case we obtain at least order one for all variables, in agreement with
the theory.

\begin{table}[htb]
    \centering
    \scriptsize
    \caption{Relative $L^2$ errors and convergence rates for the
    solutions of \eqref{eqs: Biot strong cont} against the analytical solutions.
    Results for the poroelasticity example from \Cref{ssub: numerics Poroelasticity}.}
    \label{tab:errbiothyb}
    \setlength{\tabcolsep}{4pt}
    \begin{tabular}{cc|ccccccccc}
        & & \multicolumn{9}{|c}{4F-MFEM}\\
        & $h$ & $N_{\text{dof}}$ & $\text{Err}_r$ & $\text{Rate}_r$ &
        $\text{Err}_u$ & $\text{Rate}_u$ & $\text{Err}_q$
        & $\text{Rate}_q$ & $\text{Err}_p$ & $\text{Rate}_p$ \\
        \hline
        \multirow{5}{*}{\rotatebox[origin=c]{90}{$n = 2$}}
        & 6.42e-2 & 2869   & 3.06e-2 & -    & 5.62e-2 & -    & 8.05e-2 & -    & 1.12e-1 & -    \\
        & 3.17e-2 & 10993  & 7.83e-3 & 1.93 & 2.38e-2 & 1.22 & 4.01e-2 & 0.99 & 5.67e-2 & 1.22 \\
        & 1.57e-2 & 43225  & 1.95e-3 & 1.97 & 1.13e-2 & 1.06 & 2.00e-2 & 0.99 & 2.84e-2 & 1.06 \\
        & 7.83e-3 & 171661 & 4.88e-4 & 2.00 & 5.55e-3 & 1.02 & 9.99e-3 & 1.00 & 1.42e-2 & 1.02 \\
        & 3.91e-3 & 684295 & 1.22e-4 & 2.00 & 2.77e-3 & 1.00 & 5.00e-3 & 1.00 & 7.12e-3 & 1.00 \\
        \hline
        \hline
        \multirow{5}{*}{\rotatebox[origin=c]{90}{$n = 3$}}
        & 2.34e-1 & 10988  & 1.32e-1 & -    & 1.41e-1 & -    & 2.87e-2 & -    & 4.48e-1 & - \\
        & 1.48e-1 & 42227  & 7.34e-2 & 1.29 & 6.50e-2 & 1.69 & 1.18e-2 & 1.95 & 2.60e-1 & 1.19 \\
        & 1.08e-1 & 103811 & 4.62e-2 & 1.48 & 4.09e-2 & 1.48 & 6.22e-3 & 2.04 & 1.88e-1 & 1.04 \\
        & 8.58e-2 & 203050 & 3.49e-2 & 1.21 & 3.01e-2 & 1.32 & 4.01e-3 & 1.89 & 1.46e-1 & 1.08 \\
        & 7.09e-2 & 357203 & 2.85e-2 & 1.05 & 2.42e-2 & 1.16 & 2.81e-3 & 1.87 & 1.20e-1 & 1.04\\
        \hline
        \hline
        & & \multicolumn{9}{|c}{MR-MFEM}\\
        \hline
        \multirow{5}{*}{\rotatebox[origin=c]{90}{$n = 2$}}
        & 6.42e-2 & 1572   & 3.06e-2 & -    & 5.55e-2 & -    & 8.05e-2 & -    & 1.15e-1 & -    \\
        & 3.17e-2 & 6064   & 7.83e-3 & 1.93 & 2.37e-2 & 1.21 & 4.01e-2 & 0.99 & 5.70e-2 & 0.99 \\
        & 1.57e-2 & 23928  & 1.95e-3 & 1.97 & 1.13e-2 & 1.06 & 2.00e-2 & 0.99 & 2.84e-2 & 0.99 \\
        & 7.83e-3 & 95196  & 4.88e-4 & 2.00 & 5.55e-3 & 1.02 & 9.99e-3 & 1.00 & 1.42e-2 & 1.00 \\
        & 3.91e-3 & 379822 & 1.22e-4 & 2.00 & 2.76e-3 & 1.00 & 5.00e-3 & 1.00 & 7.12e-3 & 1.00 \\
        \hline
        \hline
        \multirow{5}{*}{\rotatebox[origin=c]{90}{$n = 3$}}
        & 2.34e-1 & 5088   & 2.74e-1 & -    & 1.39e+0 & -    & 1.82e-1 & -    & 1.25e+0 & -\\
        & 1.48e-1 & 19922  & 1.26e-1 & 1.70 & 5.57e-1 & 2.01 & 8.57e-2 & 1.65 & 5.39e-1 & 1.85\\
        & 1.08e-1 & 49376  & 8.32e-2 & 1.32 & 2.93e-1 & 2.05 & 5.49e-2 & 1.43 & 3.12e-1 & 1.74\\
        & 8.58e-2 & 97049  & 6.14e-2 & 1.31 & 1.85e-1 & 1.97 & 4.17e-2 & 1.18 & 2.15e-1 & 1.61\\
        & 7.09e-2 & 171260 & 4.90e-2 & 1.19 & 1.27e-1 & 1.98 & 3.31e-2 & 1.21 & 1.61e-1 & 1.52
    \end{tabular}
\end{table}

\begin{table}[htb]
    \centering
    \scriptsize
    \caption{Relative $L^2$ errors for the
    curl of the rotation $\nabla \times r$, divergence of the displacement and flux $\nabla \cdot u$ and $\nabla \cdot q$, respectively, of \eqref{eqs: Biot strong cont} against the analytical solutions.
    Results for the poroelasticity example from \Cref{ssub: numerics Poroelasticity}.}
    \label{tab:errbiotdifferentialmulti}
    \setlength{\tabcolsep}{4pt}
    \begin{tabular}{cc|cccccc}
        & & \multicolumn{6}{|c}{4F-MFEM}\\
        & $h$
        & $\text{Err}_{\nabla \times r}$
        & $\text{Rate}_{\nabla \times r}$
        & $\text{Err}_{\nabla \cdot u}$
        & $\text{Rate}_{\nabla \cdot u}$
        & $\text{Err}_{\nabla \cdot q}$
        & $\text{Rate}_{\nabla \cdot q}$
        \\
        \hline
        \multirow{5}{*}{\rotatebox[origin=c]{90}{$n = 2$}}
        & 6.42e-2 & 3.38e-2 & -    & 1.86e-1 & -    & 2.25e-1 & -    \\
        & 3.17e-2 & 1.30e-2 & 1.36 & 9.29e-2 & 0.98 & 1.13e-1 & 0.97 \\
        & 1.57e-2 & 4.35e-3 & 1.55 & 4.66e-2 & 0.98 & 5.68e-2 & 0.98 \\
        & 7.83e-3 & 1.59e-3 & 1.45 & 2.34e-2 & 0.99 & 2.84e-2 & 1.00 \\
        & 3.91e-3 & 5.70e-4 & 1.47 & 1.17e-2 & 1.00 & 1.42e-2 & 1.00 \\
        \hline
        \hline
        \multirow{5}{*}{\rotatebox[origin=c]{90}{$n = 3$}}
        & 2.34e-1 & 3.22e-1 & -    & 6.70e-1 & -    & 7.31e-1 & -     \\
        & 1.48e-1 & 2.31e-1 & 0.72 & 3.79e-1 & 1.25 & 4.74e-1 & 0.95  \\
        & 1.08e-1 & 1.61e-1 & 1.16 & 2.76e-1 & 1.02 & 3.48e-1 & 0.99  \\
        & 8.58e-2 & 1.30e-1 & 0.91 & 2.15e-1 & 1.07 & 2.77e-1 & 0.98  \\
        & 7.09e-2 & 1.07e-1 & 1.05 & 1.77e-1 & 1.04 & 2.28e-1 & 1.01  \\
        \hline
        \hline
        & & \multicolumn{6}{|c}{MR-MFEM}\\
        \hline
        \multirow{5}{*}{\rotatebox[origin=c]{90}{$n = 2$}}
        & 6.42e-2 & 3.38e-2 & -    & 1.87e-1 & -    & 2.25e-1 & -    \\
        & 3.17e-2 & 1.30e-2 & 1.36 & 9.30e-2 & 0.99 & 1.13e-1 & 0.97 \\
        & 1.57e-2 & 4.35e-3 & 1.55 & 4.66e-2 & 0.98 & 5.68e-2 & 0.98 \\
        & 7.83e-3 & 1.59e-3 & 1.45 & 2.34e-2 & 0.99 & 2.84e-2 & 1.00 \\
        & 3.91e-3 & 5.70e-4 & 1.47 & 1.17e-2 & 1.00 & 1.42e-2 & 1.00 \\
        \hline
        \hline
        \multirow{5}{*}{\rotatebox[origin=c]{90}{$n = 3$}}
        & 2.34e-1 & 3.22e-1 & -    & 3.65e+0 & -    & 7.30e-1 & -     \\
        & 1.48e-1 & 2.31e-1 & 0.72 & 1.47e+0 & 1.98 & 4.74e-1 & 0.94  \\
        & 1.08e-1 & 1.61e-1 & 1.16 & 8.01e-1 & 1.95 & 3.47e-1 & 0.99  \\
        & 8.58e-2 & 1.30e-1 & 0.91 & 5.21e-1 & 1.85 & 2.77e-1 & 0.98  \\
        & 7.09e-2 & 1.07e-1 & 1.05 & 3.68e-1 & 1.82 & 2.28e-1 & 1.01
    \end{tabular}
\end{table}

Finally, in \Cref{tab:errbiotcomparison} and
\Cref{tab:errbiotcomparisondifferential} we compare the solutions of the two
numerical methods by computing the relative norms of their difference. As in the elasticity example, we
notice that in 2D the curl of $r$ is unaffected by the quadrature rule
\eqref{eq: quadrature rule} in accordance with \Cref{lem: invariant curl}.
For the other variables we get at least first order convergence, because both solutions converge linearly to the true solution.

\begin{table}[htb]
    \centering
    \scriptsize
    \caption{
    Relative differences in $L^2$ between the
    solutions obtained with the 4F-MFEM and MR-MFEM scheme for the rotation $r$,
    displacement $u$, flux $q$, and pressure $p$. Results for the poroelasticity example in
    \Cref{ssub: numerics Poroelasticity}.}
    \label{tab:errbiotcomparison}
    \setlength{\tabcolsep}{4pt}
    \begin{tabular}{cc|cccccccc}
        & $h$ & $\text{Err}_r$ & $\text{Rate}_r$ &
        $\text{Err}_u$ & $\text{Rate}_u$ & $\text{Err}_q$ & $\text{Rate}_q$
        & $\text{Err}_p$ & $\text{Rate}_p$
        \\
        \hline
        \multirow{5}{*}{\rotatebox[origin=c]{90}{$n = 2$}}
        & 6.42e-2 & 1.13e-13 & - & 1.35e-2 & -    & 7.66e-4 & -    & 2.23e-2 & -    \\
        & 3.17e-2 & 3.17e-13 & - & 3.37e-3 & 1.97 & 2.74e-4 & 1.46 & 5.60e-3 & 1.96 \\
        & 1.57e-2 & 9.82e-13 & - & 8.39e-4 & 1.98 & 8.20e-5 & 1.71 & 1.40e-3 & 1.97 \\
        & 7.83e-3 & 1.38e-12 & - & 2.10e-4 & 1.99 & 3.13e-5 & 1.38 & 3.52e-4 & 1.99 \\
        & 3.91e-3 & 7.31e-12 & - & 5.24e-5 & 2.00 & 1.08e-5 & 1.54 & 8.80e-5 & 2.00 \\
        \hline
        \hline
        \multirow{5}{*}{\rotatebox[origin=c]{90}{$n = 3$}}
        & 2.34e-1 & 1.58e-1 & -    & 7.88e-1 & -    & 3.72e-1 & -    & 7.47e-1 & -    \\
        & 1.48e-1 & 1.02e-1 & 0.97 & 4.80e-1 & 1.08 & 1.40e-1 & 2.14 & 4.22e-1 & 1.25 \\
        & 1.08e-1 & 7.25e-2 & 1.08 & 2.71e-1 & 1.82 & 7.83e-2 & 1.86 & 2.35e-1 & 1.87 \\
        & 8.58e-2 & 5.72e-2 & 1.03 & 1.75e-1 & 1.90 & 5.43e-2 & 1.58 & 1.51e-1 & 1.90 \\
        & 7.09e-2 & 4.67e-2 & 1.06 & 1.20e-1 & 1.96 & 4.05e-2 & 1.54 & 1.03e-1 & 1.99
    \end{tabular}
\end{table}

\begin{table}[htb]
    \centering
    \scriptsize
    \caption{Relative differences in $L^2$ between the
    curl of the rotation $\nabla \times r$, divergence of the displacement $\nabla \cdot u$, and divergence of the flux $\nabla \cdot q$ of the solutions obtained with the 4F-MFEM and MR-MFEM. Results for the poroelasticity example in
    \Cref{ssub: numerics Poroelasticity}.}
    \label{tab:errbiotcomparisondifferential}
    \setlength{\tabcolsep}{4pt}
    \begin{tabular}{cc|ccccccccc}
        & $h$
        &  $\text{Err}_{\nabla \times r}$
        & $\text{Err}_{\nabla \cdot u}$ &
        $\text{Rate}_{\nabla \cdot u}$
        & $\text{Err}_{\nabla \cdot q}$ &
        $\text{Rate}_{\nabla \cdot q}$
        \\
        \hline
        \multirow{5}{*}{\rotatebox[origin=c]{90}{$n = 2$}}
        & 6.42e-2 & 3.35e-13 & 1.63e-2 & -    & 3.58e-4 & -    \\
        & 3.17e-2 & 1.75e-12 & 4.11e-3 & 1.95 & 9.00e-5 & 1.96 \\
        & 1.57e-2 & 1.03e-11 & 1.03e-3 & 1.97 & 2.26e-5 & 1.97 \\
        & 7.83e-3 & 3.73e-11 & 2.59e-4 & 1.99 & 5.65e-6 & 1.99 \\
        & 3.91e-3 & 1.90e-10 & 6.47e-5 & 2.00 & 1.41e-6 & 2.00 \\
        \hline
        \hline
        \multirow{5}{*}{\rotatebox[origin=c]{90}{$n = 3$}}
        & 2.34e-1 & 6.35e-13 & 9.56e-1 & -    & 4.55e-3 & -    \\
        & 1.48e-1 & 9.37e-12 & 8.49e-1 & 0.26 & 1.82e-3 & 2.01 \\
        & 1.08e-1 & 1.03e-11 & 6.22e-1 & 0.99 & 9.57e-4 & 2.05 \\
        & 8.58e-2 & 3.33e-11 & 4.40e-1 & 1.50 & 6.05e-4 & 1.98 \\
        & 7.09e-2 & 5.48e-11 & 3.12e-1 & 1.80 & 4.12e-4 & 2.01
    \end{tabular}
\end{table}

\subsection{Parameter-robust preconditioning}
\label{sub:parameter_robust_preconditioning}

In this part, we present the performance of the parameter-robust preconditioner
for 4F-MFEM presented in \Cref{sec: Parameter-robust preconditioning} on the test case from \Cref{subsec:error}.
For this, we first symmetrize the system by negating the second and third rows of $\mathcal{A}$ in \eqref{eqs: Biot strong cont}. 
We then consider the MINRES iterative algorithm with stopping criteria based on
relative residual tolerance, which is set to $10^{-5}$.
In \Cref{fig:biot_plot} we present the number of iterations obtained for $n=2$
and $3$ for a wide range of the material parameters. We see that the number of iterations is
stable for most of the parameter values. We notice a slight dependency in a few
cases, but the number of iterations remains moderate.

\begin{figure}
    \centering
    \text{$n=2$}\\[0.25cm]
    \includegraphics[width=0.95\textwidth]{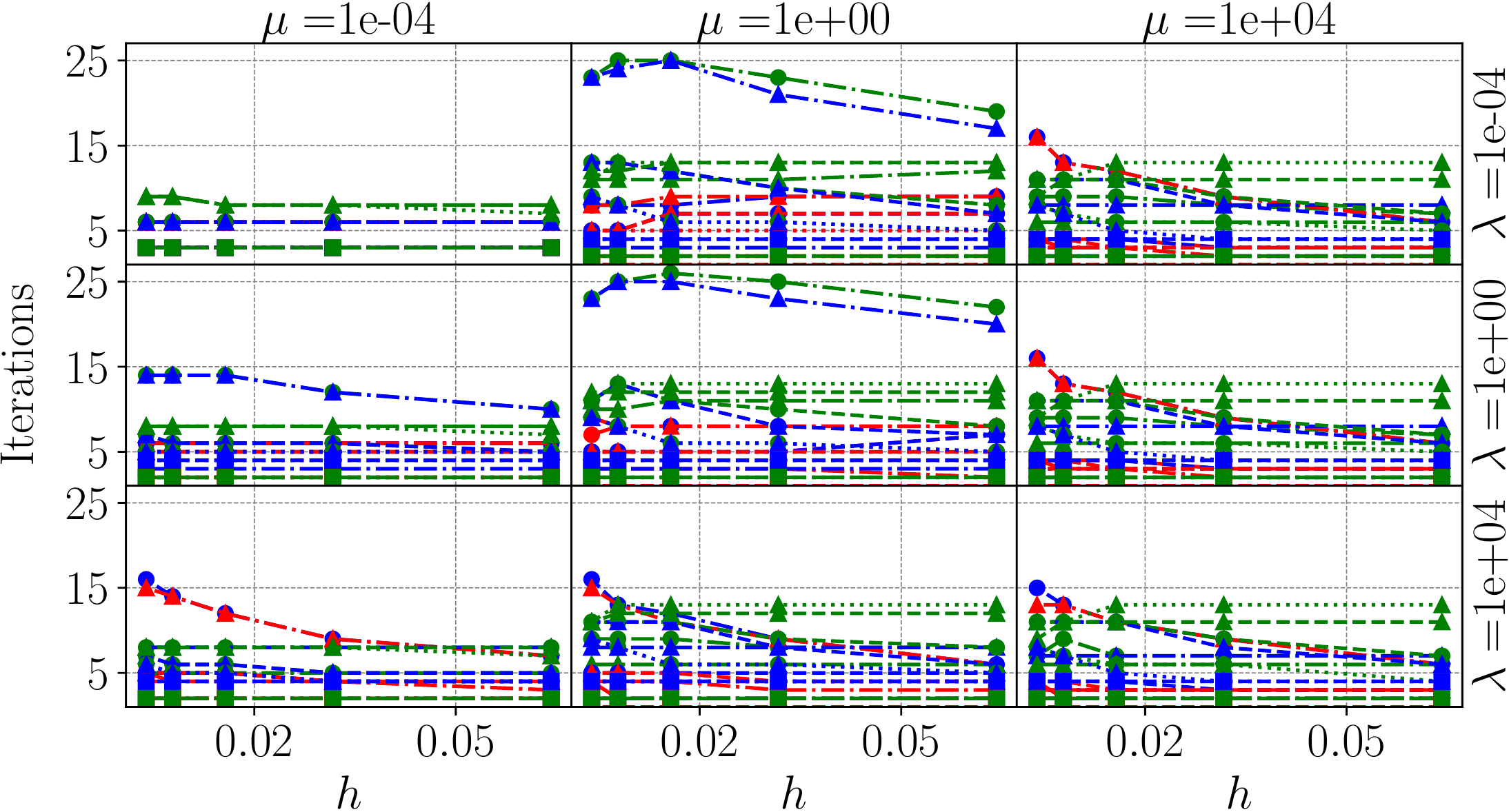}\\[0.4cm]
    \text{$n=3$}\\[0.25cm]
    \includegraphics[width=0.95\textwidth]{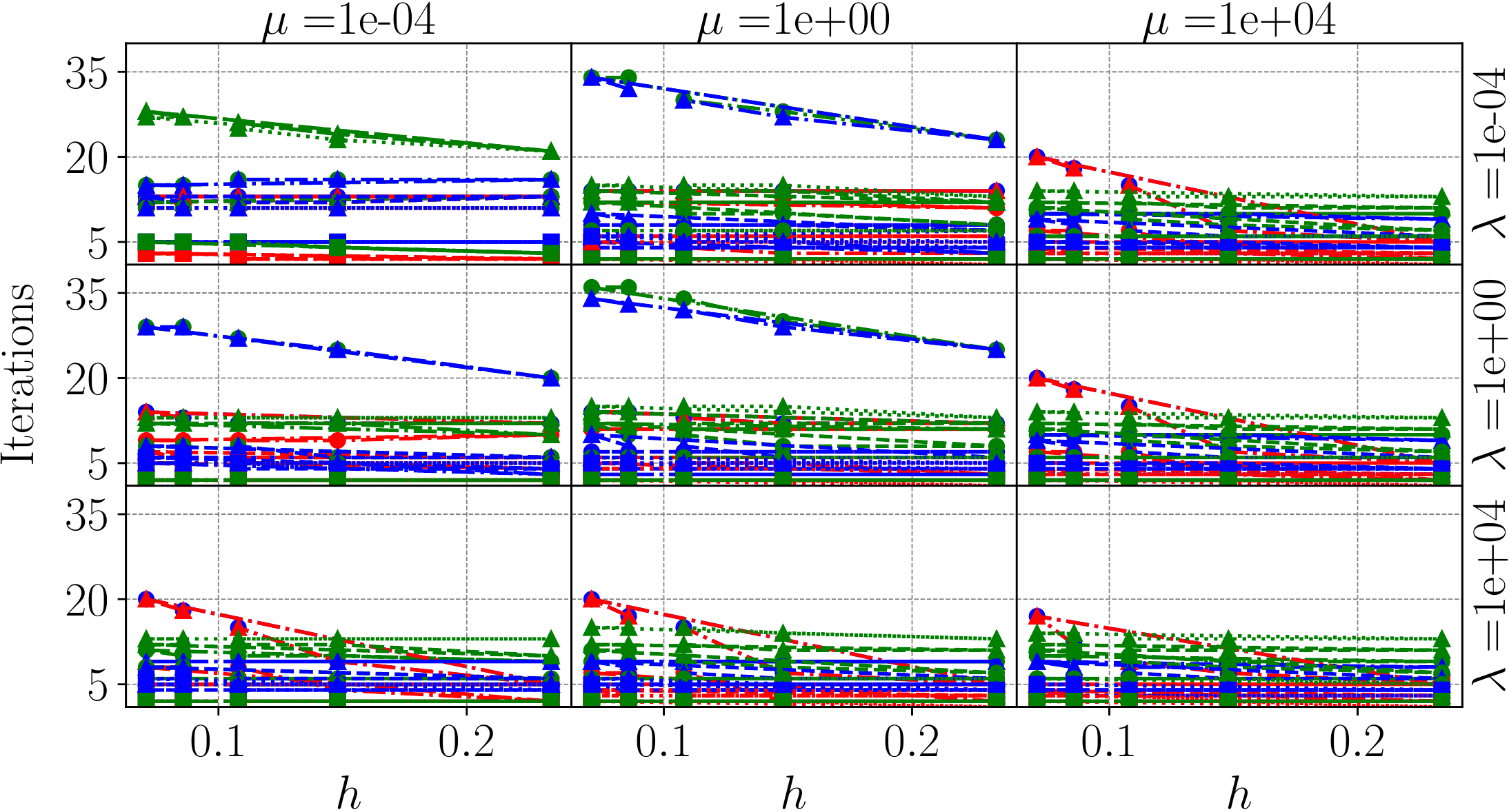}\\
    \includegraphics[width=0.65\textwidth]{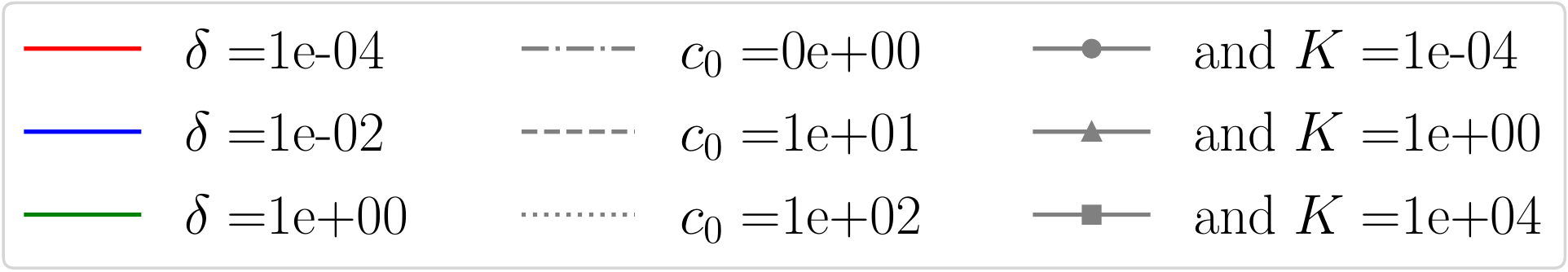}
    \caption{The number of MINRES iterations remains stable for a wide range of parameter values when applying the proposed preconditioner.}
    \label{fig:biot_plot}
\end{figure}

Finally, we have applied the same strategy for the elasticity problem
\eqref{eq:mechanics}, which can be seen as the limit case of $\alpha=0$. By varying its parameters and mesh size, we obtain a stable number of
iterations between 2 and 5 (not reported in a figure), confirming the robustness of the
preconditioner.


\subsection{Mandel's  problem}\label{subsec:mandel}

In this section we consider Mandel's test case \cite{Abousleiman1996}, which admits an
exact solution in 2D and exhibits interesting time-dependent effects. Let us
consider a two-dimensional, poroelastic domain between two rigid
plates on the upper and bottom boundaries of the domains, that are free to slide.
A force of $2F$ is applied on the plates causing a uniform displacement $\nu \cdot u$ that
is independent of $x$. The domain is displayed in the left of \Cref{dom_mandel}. Due to the symmetries of the problem, we consider a quarter of the domain.
Classically, the following boundary conditions are applied:
\begin{align*}
    &&
    \nu \cdot u  &=0, & 
    \nu \times (\sigma \nu)  &= 0, & 
    \nu \cdot q  &= 0, & 
    \text{on } &\Gamma_1 \cup \Gamma_2,  \\
    &&
    && 
    \sigma \nu &=0, & 
    p &=0, & 
    \text{on } &\Gamma_3, \\
    \nu \cdot u &= \text{const}, & 
    \frac1a \int_{\Gamma_4} \nu \cdot (\sigma \nu) &= -2F, &
    \nu \times (\sigma \nu) &= 0, &
    \nu \cdot q &= 0 & \text{on } &\Gamma_4.
\end{align*}

We adapt these conditions to our formulation of the problem by setting $\partial_u \Omega = \partial \Omega$, $\partial_p \Omega = \Gamma_3$, and $\partial_q \Omega = \partial \Omega \setminus \Gamma_3$. First, on $\Gamma_1$ and $\Gamma_2$, we set $\nu \cdot u = 0$ and, together with $\nu \times (\sigma \nu) = 0$, we obtain $r = 0$, which is set as an essential condition. On the other hand, on $\Gamma_3$ and $\Gamma_4$, we set $\nu \cdot u$ to be equal to the analytical solution and $r = 0$. The exact solution for pressure, displacement and stress can be expressed as a
series as in \cite{phillips}.

\begin{figure}[!ht]
    \centering
    \def\svgwidth{0.4\columnwidth}
\begingroup%
  \makeatletter%
  \providecommand\color[2][]{%
    \errmessage{(Inkscape) Color is used for the text in Inkscape, but the package 'color.sty' is not loaded}%
    \renewcommand\color[2][]{}%
  }%
  \providecommand\transparent[1]{%
    \errmessage{(Inkscape) Transparency is used (non-zero) for the text in Inkscape, but the package 'transparent.sty' is not loaded}%
    \renewcommand\transparent[1]{}%
  }%
  \providecommand\rotatebox[2]{#2}%
  \newcommand*\fsize{\dimexpr\f@size pt\relax}%
  \newcommand*\lineheight[1]{\fontsize{\fsize}{#1\fsize}\selectfont}%
  \ifx\svgwidth\undefined%
    \setlength{\unitlength}{449.85142505bp}%
    \ifx\svgscale\undefined%
      \relax%
    \else%
      \setlength{\unitlength}{\unitlength * \real{\svgscale}}%
    \fi%
  \else%
    \setlength{\unitlength}{\svgwidth}%
  \fi%
  \global\let\svgwidth\undefined%
  \global\let\svgscale\undefined%
  \makeatother%
  \begin{picture}(1,0.66355286)%
    \lineheight{1}%
    \setlength\tabcolsep{0pt}%
    \put(0,0){\includegraphics[width=\unitlength,page=1]{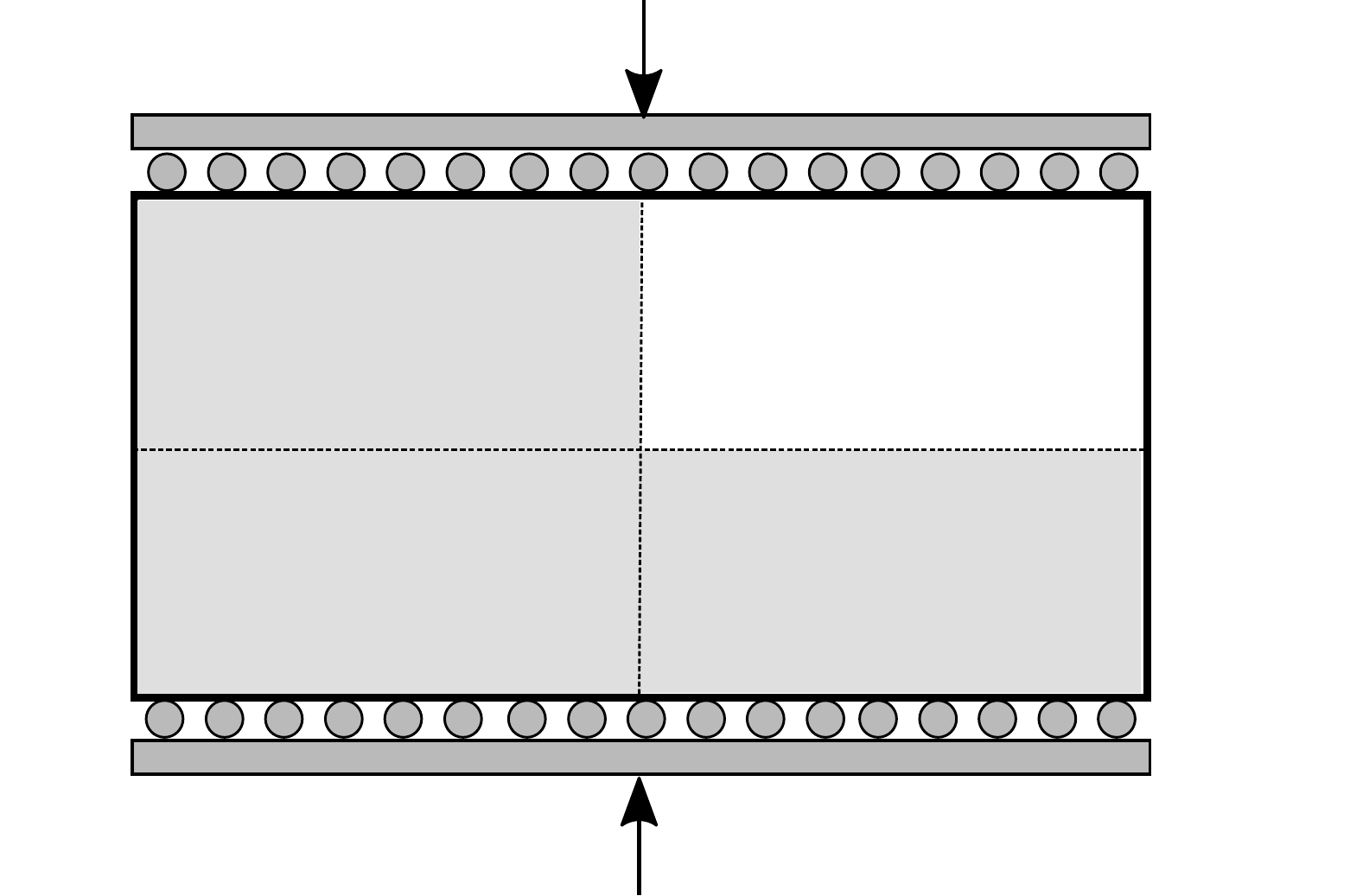}}%
    \put(0.49462947,0.62232717){\color[rgb]{0,0,0}\makebox(0,0)[lt]{\lineheight{1.25}\smash{\begin{tabular}[t]{l}$2F$\end{tabular}}}}%
    \put(0.49456506,0.02159292){\color[rgb]{0,0,0}\makebox(0,0)[lt]{\lineheight{1.25}\smash{\begin{tabular}[t]{l}$2F$\end{tabular}}}}%
    \put(0.61367311,0.38348002){\color[rgb]{0,0,0}\makebox(0,0)[lt]{\lineheight{1.25}\smash{\begin{tabular}[t]{l}$\Omega$\end{tabular}}}}%
    \put(0.6119102,0.45920121){\color[rgb]{0,0,0}\makebox(0,0)[lt]{\lineheight{1.25}\smash{\begin{tabular}[t]{l}$\Gamma_4$\end{tabular}}}}%
    \put(0.61160732,0.2754182915){\color[rgb]{0,0,0}\makebox(0,0)[lt]{\lineheight{1.25}\smash{\begin{tabular}[t]{l}$\Gamma_2$\end{tabular}}}}%
    \put(0.41066492,0.41168941){\color[rgb]{0,0,0}\makebox(0,0)[lt]{\lineheight{1.25}\smash{\begin{tabular}[t]{l}$\Gamma_1$\end{tabular}}}}%
    \put(0.85868456,0.40936195){\color[rgb]{0,0,0}\makebox(0,0)[lt]{\lineheight{1.25}\smash{\begin{tabular}[t]{l}$\Gamma_3$\end{tabular}}}}%
  \end{picture}%
\endgroup%
    \hspace*{0.1\columnwidth}%
    \raisebox{1.5cm}{
    \begin{tabular}{|rl|rl|}
    \hline
         $a=$&  100 &  $\mu =$& 2.475e9 \\ \hline
         $b=$&  10 & $\lambda =$& 1.65e9 \\ \hline
         $T=$&  5e4 & $\alpha =$&  1\\ \hline
         $\Delta t =$& 1e1 & $c_0 =$& 6.0606e-11\\ \hline
         $F =$& 6e8 & $K =$& 9.869e-11 \\ \hline
    \end{tabular}
    }
    \caption{(left) The domain for the Mandel's problem and (right)
    the material and geometric parameters.}
    \label{dom_mandel}
\end{figure}

The values of the geometric and material parameters are reported in the
table in the right of \Cref{dom_mandel}. 
\Cref{fig:sol_mandel} illustrates the numerical results obtained with both 4F-MFEM and MF-MFEM, with the analytical solutions. Excellent matching is observed for both schemes. Moreover, we obtain zero rotation, numerically, everywhere in the domain, which corresponds exactly to the zero rotation in the true solution.

\begin{figure}
    \centering
    \includegraphics[width=0.475\textwidth]{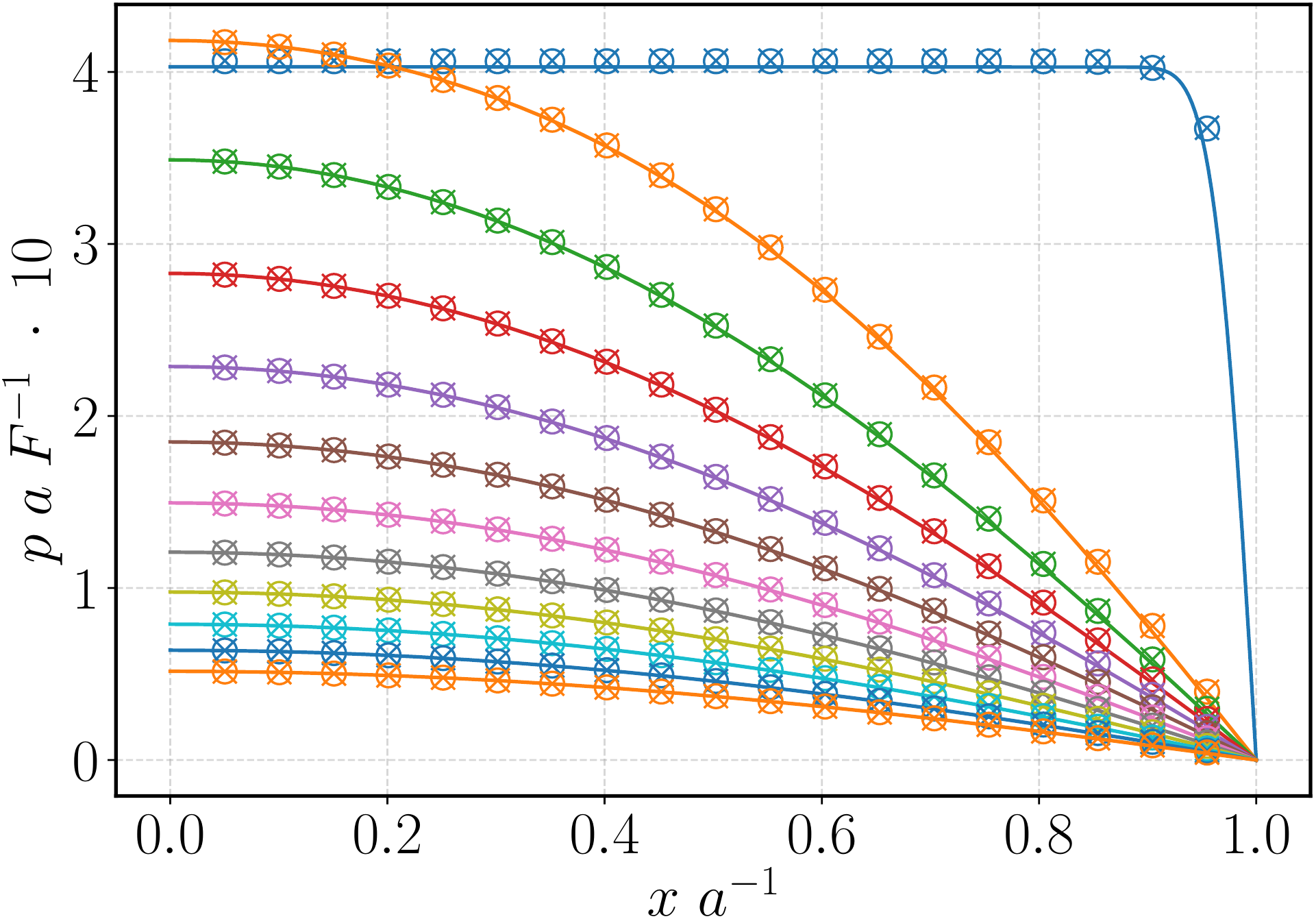}%
    \hfill%
    \includegraphics[width=0.475\textwidth]{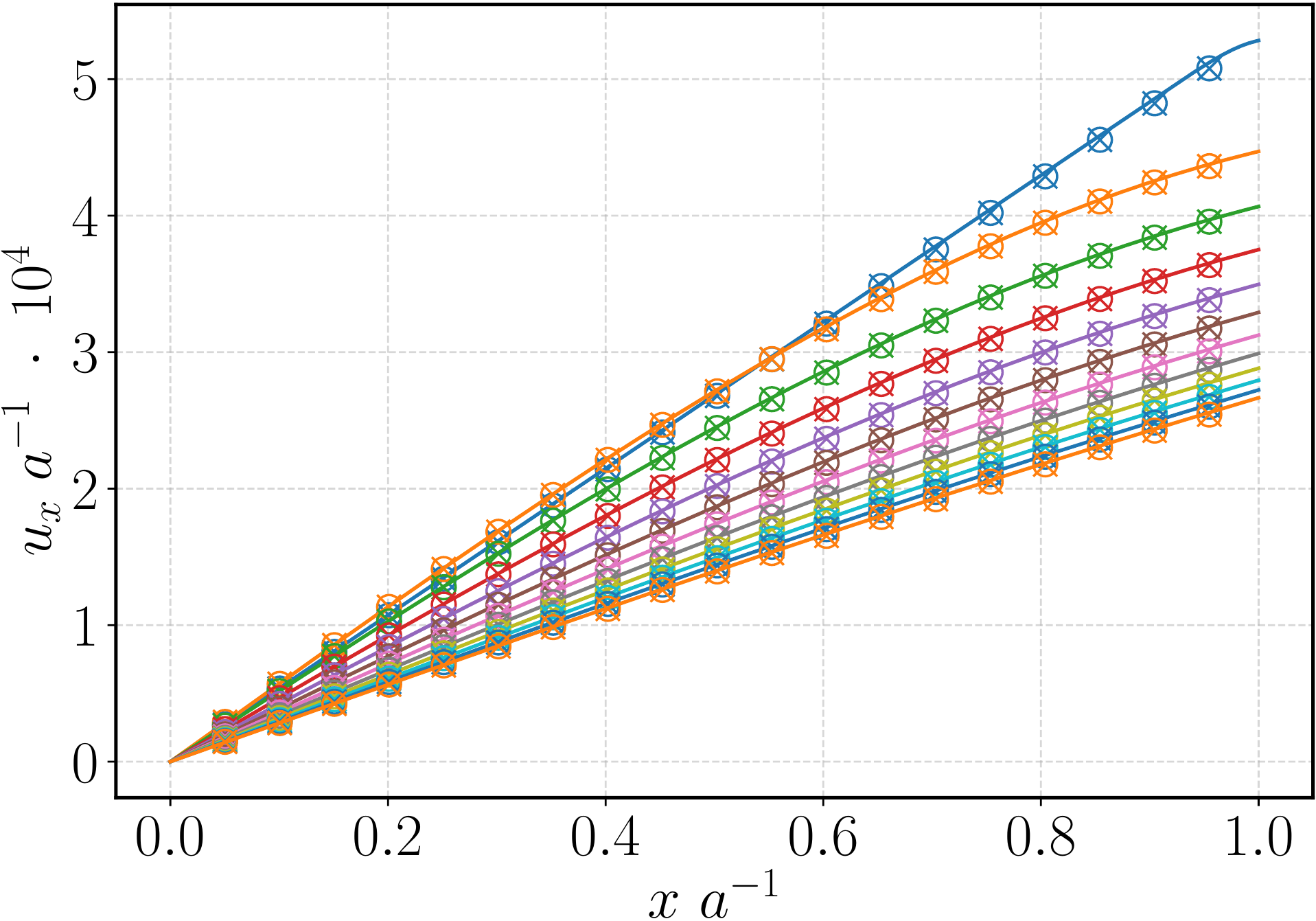}
    \includegraphics[width=0.9\textwidth]{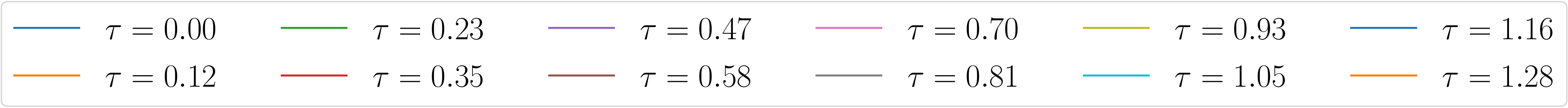}
    \caption{The scaled pressure (left) and the scaled
    displacement (right) along the $x$ axis, for different times.
    The continuous lines represent the
    analytical solution at different times $\tau$. The $\circ$ and $\times$ markers are the solutions computed
    by 4F-MFEM and MR-MFEM, respectively.}
    \label{fig:sol_mandel}
\end{figure}

\section{Concluding remarks}
\label{sec:conclusions}

We have proposed mixed finite element methods for rotation-based poroelasticity
in which the rotation variable is approximated in $H(\curl, \Omega)$. Through a
hybridization technique, the rotation and flux variables can be locally
eliminated, leading to a numerical scheme that uses $\mathbb{RT}_0 \times
\mathbb{P}_0$ for the solid displacement and fluid pressure. A priori analysis
shows that the proposed methods are stable and convergent. By using weighted
norms, we moreover derive robust preconditioners.

We remark on the limitations of our approach. First, we note that the
formulation \eqref{eq: momentum balance} of the elasticity equations as a
weighted vector Laplacian is based on a spatially constant Lam\'e parameter
$\mu$. However, the techniques from \cite{anaya2021velocity} may be applicable
for the more general case of varying $\mu$.  Second, the natural boundary
conditions for this formulation do not immediately accommodate the typical
no-stress boundary condition. In our implementation of Mandel's problem, we
therefore augmented the boundary conditions.  Third, we note that nearly
incompressible materials are not naturally handled by this formulation since
large values of $\lambda$ lead to undesirable scaling in the matrix. To capture
the incompressible limit of $\lambda \to \infty$, it may be desirable to
introduce a solid pressure variable similar to \cite{lee2019mixed}.
We aim to overcome these limitations in future work.

Finally, the numerical results presented show optimal convergence rates and,
when the preconditioner is applied, a stable number of iterations for a wide
range of data values, all in accordance with the developed
theory.

\section*{Acknowledgments}
\label{sec:acknowledgments}

The authors warmly thank Ana Budi\v{s}a, Jhabriel Varela, and Ludmil T. Zikatanov for fruitful discussions.

\bibliographystyle{siamplain}
\bibliography{references}

\end{document}